\documentclass[11pt,a4paper]{article}
\usepackage{amsfonts,amsmath,amssymb,amsthm}
\usepackage[utf8]{inputenc}  
\usepackage[T1]{fontenc}     
\usepackage{hyperref}
\usepackage{xcolor}

\newcommand{\R}{\mathbb{R}}
\renewcommand{\d}{\mathrm{d}}
\newcommand{\ddt}{\frac{\d}{\d t}}

\newtheorem{theorem}{Theorem}[section]
\newtheorem{proposition}{Proposition}[section]
\newtheorem{definition}{Definition}[section]
\newtheorem{lemma}{Lemma}[section]
\newtheorem{assumption}{Assumption}[section]
\newtheorem{remark}{Remark}[section]

\title{On the Fisher infinitesimal model without variability}
\author{Amic Frouvelle\thanks{CEREMADE, CNRS, Université Paris Dauphine--PSL, 75016 Paris, France.}
\thanks{The work of A.F. has been (partially) supported by the Project EFI ANR-17-CE40-0030 of the French National Research Agency.}
\footnotemark[5]
\and Cécile Taing\thanks{Laboratoire de Math\'ematiques et Applications, Universit\'e
de Poitiers, CNRS, F-86073 Poitiers, France.}
\thanks{C.T. has received funding from the PEPS JCJC 2021 and 2022 programs of the Insmi (CNRS Mathématiques).}
\thanks{Emails: frouvelle@ceremade.dauphine.fr, cecile.taing@univ-poitiers.fr}
}

\begin{document}
\maketitle

\begin{abstract}
  We study the long-time behavior of solutions to a kinetic equation inspired by a model 
  of sexual populations structured in phenotypes. The model 
  features a nonlinear integral reproduction operator derived from the Fisher infinitesimal operator
  and a trait-dependent selection term. The reproduction operator 
  describes here the inheritance of the mean parental traits to the offspring without variability.
  We show that, under assumptions on the growth of the selection rate,
  Dirac masses are stable around phenotypes for which the difference between the 
  selection rate and its minimum value is less than~$\frac{1}{2}$.
  Moreover, we prove the convergence in some Fourier-based distance of the centered and rescaled 
  solution to a stationary profile under some conditions on the initial moments of the solution.
\end{abstract}

\section{Introduction}

\subsection{The population model}

We study the large time behavior of solutions to a kinetic equation that is 
inspired by the modeling of population dynamics in the context of evolutionary biology. 
More precisely, this equation models the evolution of a sexual population structured by a phenotypic trait, represented 
by a continuous variable~$x\in \mathbb{R}$, that is inherited according to 
the mix of parental traits described by a collision operator.
Denoting by~$f(t,\cdot)$ the population density at time~$t \geqslant 0$ in the trait space, 
the model we are interested in is the following :
\begin{equation}
  \begin{cases}
    \partial_t f(t,x) = B_0 [f(t, \cdot)] (x)-m(x)f(t,x), \quad t \geqslant 0, \\
    f(0,x) = f^0(x),
  \end{cases}
  \label{model}
\end{equation}
where~$B_0$ is the trait inheritance operator, defined as, for~$x \in \mathbb{R}$,
\begin{equation}\label{B_operator}
  B_0 [f] (x) := \iint_{\mathbb{R}^2} \delta_0 \left( x - \frac{z_1 + z_2}{2} \right) f(z_1)
  \frac{f(z_2)}{\int_\mathbb{R} f(z') \,\d z'} \,\d z_1 \,\d z_2,
\end{equation}
with~$\delta_0$ the Dirac measure.
We interpret this last quantity as a the number of newborns with trait~$x$ per unit of time,
and thus we refer to it as the reproduction term.
With the operator~$B_0$, it is assumed that newborns inherit exactly the mean of the parental traits~$\frac12(z_1 + z_2)$.
This mixing operator also features a normalization by the total mass of the population density~$\int_{\mathbb{R}} f(z') \d z'$, to illustrate 
that the choice of a partner is made at constant rate in time, uniformly among the whole population.
We assume the selection rate~$m$ to be bounded below. 

In the present work, we investigate the long time behavior of~$f$. More precisely, we prove the 
local stability of some particular Dirac masses, and the convergence of a rescaled formulation 
of~$f$ to a stationary state with a measure-adapted Fourier distance.

\bigskip

The model~\eqref{model} is motivated by works on the Fisher infinitesimal model~\cite{Bart-Ethr-Vebe-2017,fisher1919}, 
with the following operator~$B_\varepsilon$ :
\begin{equation*}
  B_\varepsilon [f] (x) := \frac{1}{\varepsilon \sqrt{\pi}}\iint_{\mathbb{R}^2} \exp \left[ -\frac{1}{\varepsilon^2} \left(x - \frac{z_1 + z_2}{2} \right)^2 \right] f(z_1)
  \frac{f(z_2)}{\int_\mathbb{R} f(z') \,\d z'} \,\d z_1 \,\d z_2.
\end{equation*}
which describes the offspring traits as normally distributed around the mean of the parental traits 
with small variance of order~$\varepsilon^2$. This model has been used in theoretical 
evolutionary biology~\cite{bulmer}.
From a mathematical point of view, the Fisher infinitesimal operator appears in the works~\cite{MirRao,Rao2017}, 
where a different scaling than the small variance is used, and a spatial structure is added
to study invasions. The authors 
show a derivation of the Kirkpatrick-Barton system~\cite{KirBar} at a limit of a large reproduction rate.
Also the infinitesimal operator is included in a selection-competition model in~\cite{dekens-2021} with a population evolving between
two habitats and a regime of small phenotypic variance compared to the environment heterogeneity.
In~\cite{Raoul-2021}, the Fisher infinitesimal operator is combined with a multiplicative selection one.
The author showed local uniqueness of a steady state and exponential convergence of the solution
for weak selection effects, thanks to a contraction property of the reproduction operator for the~2-Wasserstein distance,
under the assumption of a compactly supported selection rate.
A time-discrete version of the Fisher infinitesimal model, with a quadratic selection rate and 
non-overlapping generations, is analyzed in~\cite{Cal-Lep-Poy-2021,Cal-Poy-San-2023},
within a framework other than the Wasserstein one.

The study of the asymptotic behavior of the Fisher infinitesimal model 
with an appropriate time scaling has been tackled in~\cite{calvez2018asymptotic, patout2020cauchy}. The authors
study the concentration of the population distribution around some particular traits. More precisely, 
in the first work the authors studied the special stationary states at a regime of small variance, while in the 
second one the author considered the associated Cauchy problem in the same regime, showing that solutions 
can be approximated by Gaussian profiles with small variance. The asymptotic framework was built on the spirit 
of~\cite{Die-Jab-Mis-Per-2005} with small variance limit for asexual population models.

Also, similar operators as~$B_\varepsilon$ have been studied in different contexts : alignment~\cite{Deg-Fro-Rao-2014, Ayo-Bru-Thi-2023}, protein exchanges~\cite{MagRao2015}, among other works.
Mathematical models of sexual populations have recently received some attention, with variations 
of the reproduction term, to account for mating preferences~\cite{Cor-et-al-2018,Lem-2018},
asymmetrical inheritance~\cite{Per-Stru-Tai-2022} and allelic structure~\cite{Col-Mel-Met-2013, Dek-Mir-2021}.

\bigskip

In the present study, we examine a version of the Fisher infinitesimal model with no variance. 
This model could be obtained through asymptotics different from those providing the Fisher infinitesimal one, 
using the law of large numbers instead of 
the central limit theorem in the number of loci.
The main ingredients of the analysis, which come from kinetic theory, are the 
estimation on the evolution of moments and the Fourier distance for probability measures. 
Indeed, the reproduction operator~$B_0$ has already been considered 
in ~\cite{Bis-Car-Tos-2005, Dur-Mat-Tos-2009, Mat-Tos-2008, Par-Tos-2006,Pul-Tos-2004} for kinetic models with dissipative collisions, though conserving the total mass. These models describe, with more general collision weights, the time behavior 
of granular gases or of wealth distribution among agents in an econometric context. In the latter case, the selection rate 
could describe a wealth-dependent involvement in trades.
The common feature of these models is the non-conservation of energy entailed 
by the collisional interactions. It has been proved that solutions to these models, after being 
rescaled according to the non-conserved energy, converge to limit profiles with overpopulated tails 
(except for very specific cases where we can observe thin tails) in the sense of Fourier distances.

More recently, the Fourier 
distance has been used in \cite{Ayo-Bru-Thi-2023} to prove Dirac concentration in a related alignment model.
Other distances have been proved to be non-expanding for measure solutions 
to structured populations PDE models, as Monge-Kantorovich type ones (see \cite{Fou-Per-2020,Fou-Per-2021}).

The idea of using the Fourier transform for Boltzmann-type equations goes back 
to the works of A. V. Bobylev in~\cite{Bob-1975,Bob-1988} and the resulting equation has also been studied in~\cite{Pulvirenti-thesis}.
Then, the Fourier distance has been employed in~\cite{Gab-Tos-Wen-1995} 
to investigate the trend to equilibrium of the
solutions to the Boltzmann equation for Maxwell molecules.
Indeed, by the use of the Fourier transform of the collision operators, the equation on~$\widehat{f}$ becomes simpler than 
the one for~$f$. This is also the case here for the operator~$B_0$ defined in~\eqref{B_operator} since we have, when~$f$ is a probability measure, 
\begin{equation*}
  \widehat{B_0(f)}(\xi)=\widehat{f}(\xi/2)^2,
\end{equation*}
which makes the use of Fourier metrics possible. 
In the mentioned works on kinetic models for granular gases or wealth distribution, 
the proof of the convergence of the solutions strongly relies, on the one hand, on a contraction 
property of the collision operator in the Fourier distance, and on the other hand,
on the estimation of some moments of the solutions.
One of the main useful properties of these models is that, even though the energy (or center of mass) is not conserved by the evolution, its behaviour is explicitly given by an exponential function of time, the rate of which is given by the parameters of the model, independently of the initial condition.

Though our trait mixing operator~$B_0$ is similar to the gain part of the collision operator in the works mentioned 
above, our model mainly differs in the selection component $m(x)$. This component is described by a trait-dependent mortality rate and is 
considered as being a constraint on the phenotypic variability of the population. When this selection component~$m(x)$ is constant, the model exactly corresponds to one of the models of~\cite{Par-Tos-2006}, therefore if the population is concentrated around a location where the selection does not vary significantly, we may expect to obtain similar results. However, the heterogeneity of this selection component strongly affects the properties 
of the solutions : there is no more conservation of mass, neither of the center of mass, and the evolution equations for the moments are not closed (in contrast to~\cite{Par-Tos-2006} for which they are explicitly computable). For instance, we cannot easily compare two different rescaled solutions, since their time rescaling (depending on the variance) are not explicitly the same.  
In the present work, we take the same approach as in the previous works with, in addition, the 
estimation of quantities derived from the selection rate, which contain the main difficulties of the consequent analysis.

The main challenge we are faced to is to obtain a criterion of stability which allows 
to prove the convergence to a concentrated profile (in contrast to~\cite{Par-Tos-2006} for which the convergence is global). Once the convergence is established, 
we may then expect the model to behave as if the selection rate was constant, and the 
method using Fourier metrics as in~\cite{Par-Tos-2006} can indeed be adapted to prove that the rescaled profile 
converges to the same heavy-tailed distribution. We are able to prove the local stability 
in Wasserstein metric of Dirac masses located at positions for which the difference 
between the selection rate and its minimum is less than~$\frac12$. This has to 
be compared with the results of~\cite{patout2020cauchy} and of the recent preprint \cite{Gue-Hil-Mir-2023} recovering the same results using a moment-based approach, similar to ours (for the Fisher infinitesimal model with selection, with small~$\varepsilon>0$). In these works, one of the criterion of stability is that the initial profile is close to a gaussian profile of variance~$\varepsilon^2$, centered at a point for which the difference between the selection rate and its minimum is less than~$1$. We find suprising that we do not obtain the same difference of selection rate, as this threshold of~$\frac12$ seems natural in the case~$\varepsilon=0$, in view of the estimates we can obtain in the differential inequalities on the moments of the solution, indicating that the basin of attraction may become smaller as the difference with the minimum of the selection rate approaches~$\frac12$. The investigation of the possibility of instability above this threshold of $\frac12$, and the comparison between the cases~$\varepsilon>0$ and~$\varepsilon=0$ are left to future work.

\bigskip

The work is organized as follows. In Section~\ref{section-general-results}, we establish the existence and uniqueness of 
the solution of~\eqref{model} in the sense of measures, we also define the moments of the centered and normalized solution and 
state some properties that they satisfy. Finally, we also derive the equation on the Fourier transform of the centered and rescaled
formulation of the solution. In Section~\ref{section-moments}, we provide a local stability result 
of Dirac masses and derive estimates 
on the exponential decay in time of the moments. In Section~\ref{section-4}, we prove the 
long time convergence of the centered and rescaled solution to a stationary self-similar profile.

\subsection{Assumptions and main results}

We outline in this section the assumptions we use and the main results that we prove throughout the present work.

\bigskip

First, we state the existence and uniqueness of the solution to~\eqref{model} in the sense of measures. 
Denoting~$\mathcal{M}_+$ the set of nonnegative Borel measures with finite mass, we use the following 
definition of a measure solution.
\begin{definition}\label{def-weak-solution}
  If~$f^0\in\mathcal{M}_+$, we say that~$f\in C([0,T],\mathcal{M}_+)$ is a weak solution to the model~\eqref{model} if, for any Borel set~$A\subset\mathbb{R}$, we have for all~$t\in[0,T]$:
  \begin{equation}
    \label{eq-weak-solution}
    \int_Af(t,x)\d x=\int_Ae^{-m(x)t}f^0(x)\d x + \int_0^t\int_Ae^{-m(x)(t-s)}B_0[f(s,\cdot)](x)\d x\, \d s.
  \end{equation}
\end{definition}

The existence and uniqueness result is given by the following theorem that we prove in Section~\ref{section-exist-uniq}.
\begin{theorem}\label{thm-existence-uniqueness}
  If~$m$ is a measurable function and is bounded below, then for any~$f^0$ in~$\mathcal{M}_+$, and for any~$T>0$, there exists a unique weak solution~$f\in C([0,T],\mathcal{M}_+)$ to the model~\eqref{model}, in the sense of Definition~\ref{def-weak-solution}.
\end{theorem}
The proof of Theorem~\ref{thm-existence-uniqueness} relies on standard arguments of fixed point construction using 
a distance based on the total variation of a measure.

\bigskip 

Next, we focus on the study of the moments of~$f$. 
As a first step, we derive a control on the time propagation of the initial moments of~$f$,
which enables to define the moments of~$f$ at any time.
Namely, defining for any finite nonnegative measure~$\nu$, when~$\int_\mathbb{R}|x|^k\nu(x)\d x$ is finite, its moment of order~$k$ by
\begin{equation*}
  \mu_k(\nu):=\int_\mathbb{R}x^k\nu(x)\d x,
\end{equation*}
the moments of~$f(t,\cdot)$, denoted by~$\mu_k(f(t,\cdot))$ are well defined for all time as soon as they are finite initially.
Moreover, we prove that these moments are continuous and, for~$k$ even, 
that they satisfy some kind of differential inequalities,
which will be used to establish estimates on the centered moments of the normalized density, denoted by
\begin{equation*}\label{def-g}
  g(t,x) = \frac{f(t,x)}{\int_\R f(t,z)\,\d z}.
\end{equation*}

We aim to get some regularity on the centered moments of~$g$ in order to determine 
the long time behavior of its centered and standardized formulation, which we will also call self-similar profile throughout this work.
For this purpose, we add the following assumption of the selection rate~$m$.
\begin{assumption}\label{assumption-growth}
  We suppose that~$m$ is measurable and that there exist constants~$K\geqslant0$ and~$C\geqslant0$ such that for all~$x\in\mathbb{R}$, we have
  \[-K\leqslant m(x)\leqslant C(1+x^2).\]
  \end{assumption}
A first consequence of this assumption is that
the time derivative of the moment of order~$k$ of~$f$ can be expressed with the lower order moments  at any time, 
provided that the initial data has a finite moment of order~$k+2$. This implies that, in this case, the moments of~$f$
are continuously differentiable. 

Note that we do not impose the selection rate to be nonnegative.
This may be interpreted as a supplementary asexual reproduction rate favored by the environment.
Indeed, we could decompose $m(x)$ under the fom $m(x) = d(x) - b(x)$ with $b(x)$ and $d(x)$ defined 
as, respectively, birth and death rates, and assumed to be nonnegative. Thus, under the assumption 
of neglected mutations, the birth rate $b(x)$ entails a multiplicative operator, as the death rate, 
which leads to a selection rate $m(x)$ that may change sign. Assumption~\ref{assumption-growth} therefore means that we assume the birth rate to be bounded and the death rate to grow at most quadratically.

Then, we define the centered moments of~$g$ and the centered variations of~$m$ by the quantities
\begin{align}
  M_k(t)&=\int_\mathbb{R}\big(x-\overline x(t)\big)^kg(t,x)\d x,\label{def-Mk}\\
  S_k(t)&=\int_\mathbb{R}\big(x-\overline x(t)\big)^k\big(m(x)-m(\overline x(t))\big)g(t,x)\d x,\label{def-Sk}
\end{align}
with~$\overline{x} (t)$ the center of mass at time~$t$,
\begin{equation}\label{def-xbar}
  \overline{x}(t) := \int_\mathbb{R} x g(t,x)\,\d x = \frac{\mu_1(f(t,\cdot))}{\rho(t)},
\end{equation}
and we denote their initial values by~$M_k^0$,~$S_k^0$,~$\overline{x}^0$.
We obtain differential equations on the centered moments of~$g$ for any~$k\geqslant 2$, assuming that~$f^0$ has a finite moment of order~$k+2$.
In the case where~$k\geqslant4$ is even and~$f^0$ has only a moment of order~$k$, we get a differential inequality on~$M_k$.

A quantity that appears to be important in our study of these differential inequalities is the shifted deviation of~$m$ from its minimum, defined as
\begin{equation}
  \eta(x)=\inf_{\mathbb{R}}m + \frac12-m(x),
  \label{def-eta}
\end{equation}
which is a quantity not exceeding~$\frac12$ (and which is positive if the excess of mortality is less than~$\frac12$).

For instance, when~$f$ has initially a fourth moment, we obtain
\begin{align}
  \label{dtxbar}\frac{\d}{\d t}\overline{x}&=-S_1 ,\\
  \label{dtM2}\frac{\d}{\d t}M_2&=-\frac12M_2-S_2+S_0M_2,\\
  \label{dtM4}\frac{\d^+}{\d t}M_4&\leqslant-\Big(\frac38+\eta(\overline{x})-S_0\Big)M_4+\frac38M_2^2+4S_1M_3,
\end{align}
where we have written~$\frac{\d^+}{\d t}M_4(t):=\limsup_{\tau\to0,\tau>0}\frac{M_4(t+\tau)-M_4(t)}{\tau}$.

One can see that we always have~$S_2\geqslant(\eta(\overline{x})-\frac12)M_2$, and therefore the differential equation~\eqref{dtM2} for~$M_2$ gives~$\frac{\d}{\d t}M_2\leqslant-(\eta(\overline{x})-S_0)M_2$. We expect the variation~$S_0$ to be small, and we will see that a criterion for exponential decay of~$M_2$ is indeed that~$\eta(\overline{x})$ is initially positive.

  With these differential equations and inequations, we aim to derive time estimates on the centered moments~$M_k$.
To do so, we need to have some bounds concerning the regularity and growth of the selection rate, which we state in the 
following assumption, which is for instance satisfied when~$m$ is~$C^2$ with bounded second derivative, or when~$m$ is globally Lipschitz. Notice that this allows to cover the cases of smooth selection rates with quadratic growth.

  \begin{assumption} \label{assumption-local-lipschitz} We suppose that~$m$ is locally Lipschitz and satisfies the inequalities of Assumption~\ref{assumption-growth}.
  \end{assumption}

Under this assumption, we prove the stability of Dirac masses around positions~$\overline{x}^0$ for which the excess of mortality compared to its minimum value is less than~$\frac12$, that is to say~$\eta(\overline{x}^0)>0$. We show that if the center of mass at initial time is at such a position and if~$M_4^0$ is small enough (which means that the initial normalized profile~$g^0$ is close to a Dirac mass), then the center of mass converges to some limit, and the moments~$M_2$ and~$M_4$ decrease exponentially to~$0$ at long time, as stated in the following theorem.

\begin{theorem}\label{thm-stability-dirac}
  Under Assumption~\ref{assumption-local-lipschitz}, we suppose that initially we have~$\eta(\overline{x}_0)>0$. Then, for all~$\delta$ such that~$0<\delta<\eta(\overline{x}_0)$, if~$M_4^0$ is sufficiently small, we have for all~$t>0$:
  \begin{gather}
    \eta(\overline{x}(t))\geqslant \delta,\label{eta-xbar-delta0}\\
    M_2(t)\leqslant M_2^0e^{-\delta t},\label{M2-decay-delta2}\\
    M_4(t)\leqslant M_4^0e^{-\delta t}.\label{M4-decay-delta2}
  \end{gather}
  Furthermore, in that case,~$\overline{x}(t)$ converges exponentially fast towards some~$\overline{x}_\infty\in\mathbb{R}$.
  Consequently it means that if~$g$ is initially sufficiently close to a Dirac mass located at~$\overline{x}_0$ (in~$4$-Wasserstein distance), then it converges exponentially fast to a Dirac mass at some point~$\overline{x}_\infty$, and the closer~$g$ is from~$\delta_{\overline{x}_0}$, the smaller the distance between~$\overline{x}_0$ and~$\overline{x}_\infty$.
\end{theorem}

As the solution concentrates around a point in space, we expect the behaviour of the solution to be close to the case in which the selection rate is constant. But in that case it is known that the second moment decays as~$e^{-\frac{t}{2}}$, which is much stronger that the decay~$e^{-\delta t}$ provided by Theorem~\ref{thm-stability-dirac}, when~$\delta$ is small. However, with more work, we are able to improve the rates of convergence of the moments. First, we start by refining the estimates on~$M_2$ and~$M_4$ alone thanks to~\eqref{dtM2}-\eqref{dtM4}, then we obtain better rates of convergence for higher order moments, which finally are propagated back in the differential equations for lower moments to obtain precise estimates for all moments. In the end we obtain the following improved estimates.

  \begin{proposition}\label{prop-improved-estimates-intro}
    We suppose that the assumptions of Theorem~\ref{thm-stability-dirac} are satisfied and that, for some~$k_0\geqslant2$, the initial moment~$M_{2k_0}^0$ is finite.
    
    Then for all~$k\leqslant k_0$ and~$\lambda<\min(1-\frac1{2^{2k-1}},\frac12-\frac1{2^{2k_0-1}}+\delta)$, there exists~$C_{M^0_{2k_0}}$ (converging to~$0$ when~$M^0_{2k_0}\to0$) such that for all~$t>0$, we have
    \[M_{2k}\leqslant C_{M^0_{2k_0}}e^{-\lambda t}.\]
\end{proposition} 
Notice that here, even for a~$\delta$ very small, if a sufficiently high initial moment is finite, then we recover the decay as~$e^{-\frac{t}2}$ for the second moment~$M_2$.

\bigskip

Now that we know that this second moment decays as in the case of constant selection rate, we can go further by studying
the Fourier transform of the centered and rescaled formulation of~$g$ and then characterize 
its convergence to a limit profile at long time, as announced in the beginning of this work.

Indeed, in order to investigate the long time behavior of the solution to~\eqref{model}, we make use of the following Fourier distance
\begin{equation}\label{def-fourier-distance}
d_s(\gamma_1,\gamma_2) := \sup_{\xi\neq0}\frac{|\widehat{\gamma_1}(\xi)-\widehat{\gamma_2}(\xi)|}{|\xi|^s},
\end{equation}
which is defined when~$\gamma_1$ and~$\gamma_2$ are some measures that have same moments up to order~$p \in \mathbb{N}$, with respective Fourier transforms~$\widehat{\gamma_1}$ and~$\widehat{\gamma_2}$,
and~$s\in(p,p+1]$ to be determined. 

For all~$\xi \in \R$, we can define the Fourier transform of~$f$ solution to~\eqref{model} on the trait variable
\begin{equation*}
  \widehat{f}(t, \xi):=\widehat{f(t, \cdot)} (\xi) = \int_\R e^{-i\xi x} f(t,x)\d x.
\end{equation*}
Thus, to apply the Fourier distance~\eqref{def-fourier-distance} on measures that have same moments up to order 2,
we define~$\gamma(t, \cdot)$ as the rescaled and centered formulation of~$g(t, \cdot)$, or associated self-similar profile,
\begin{equation}\label{def-gamma}
  \gamma(t,x) := \sqrt{M_2(t)} \, g\left(t, \sqrt{M_2(t)}x + \overline x (t)\right).
\end{equation}
Then, we obtain that the Fourier transform satisfies the equation 
\begin{equation*}
  \partial_t\widehat{\gamma}(t,\xi)=\widehat{\gamma}(t,\tfrac{\xi}2)^2-\widehat{\gamma}(t,\xi)+\frac14\xi\partial_\xi\widehat{\gamma}(t,\xi)+R(t,\xi),
\end{equation*}
in which the term~$R(t,\xi)$ comprises all the quantities generated by the selection rate~$m$. Thus,~$R$
depends on some of the moments~$M_k$ and~$S_k$, because of the formulation of~$\gamma$ and of the equations~\eqref{dtxbar} and~\eqref{dtM2}.

So, this last equation on~$\widehat{\gamma}$ is divided into an operator applied only on~$\widehat{\gamma}$, which comes from 
the Fourier formulation of the reproduction operator~$B_0$ for probability distributions, and the remaining term~$R$ that 
we expect to become small at long time, because of the stability result on the center of mass of~$g$ and  
the estimates on the moments stated in Theorem~\ref{thm-stability-dirac}.

In fact, it appears that the estimates on~$M_k$ obtained at this point, even with the improvements provided by Proposition~\ref{prop-improved-estimates-intro}, are still not sufficient to control the remaining term~$R$.  
More precisely, we need to specifically quantify the exponential decay of the moment~$M_2$ by providing a lower bound, and
also to derive a control on the ratios~$\frac{M_{4}}{M_2}$ and~$\frac{M_6}{M_2}$. To this aim, we make a stronger assumption on the initial condition, by starting with an initial profile~$g^0$ such that~$\frac{M^0_{2k_0}}{M^0_2}$ is small (for some~$k_0$ large enough). This is not necessarily the case for any profile close to a Dirac mass, but if we take any profile centered around~$\overline{x}^0$ (with finite moment of order~$2k_0$) and we shrink it by some large factor, we satisfy this assumption. In this framework we are able to obtain precise rates of convergence, as follows.

\begin{theorem}\label{thm-moments-decay}
  Under the assumptions of Theorem~\ref{thm-stability-dirac}, we take~$k_0\geqslant2$ such that~$\frac1{2^{2k_0-1}}<\delta$, and we suppose that~$\frac{M_{2k_0}^0}{M_2^0}$ is sufficiently small. Then,
  \begin{itemize}
  \item there exists~$C_{k_0}\big(\frac{M_{2k_0}^0}{M_2^0}\big)>0$, converging to~$1$ as~$\frac{M_{2k_0}^0}{M_2^0}\to0$ such that
    \[M_2(t)\geqslant C_{k_0}\big(\tfrac{M_{2k_0}^0}{M_2^0}\big)M_2^0\,e^{-\frac{t}2},\]
  \item for all~$k\geqslant2$ (with~$k\leqslant k_0$) and~$\lambda<\min(\frac12-\frac1{2^{2k-1}},\delta-\frac1{2^{2k_0-1}})$, there exists~$\widetilde{C}_{\lambda,k}\big(\frac{M_{2k_0}^0}{M_2^0}\big)>0$, converging to~$0$ as~$\frac{M_{2k_0}^0}{M_2^0}\to0$ such that
    \[\frac{M_{2k}(t)}{M_2(t)}\leqslant\widetilde{C}_{\lambda,k}\big(\tfrac{M_{2k_0}^0}{M_2^0}\big)\,e^{-\lambda t}.\]
  \end{itemize}
\end{theorem}

\bigskip

Finally, the last theorem states the convergence in the Fourier distance of the profile~$\gamma$ to a stationary 
profile~$\gamma_\infty$, which is identified in~\cite{Bal-Mar-Pug-2002, Car-Tos-2007},  when~$t$ goes to~$+\infty$, as
\begin{equation}\label{def-gamma-inf}
  \gamma_\infty(x)=\frac{2}{\pi(1+x^2)^2},
\end{equation}
or given by
\begin{equation*}
  \widehat{\gamma_\infty}(\xi)=(1+|\xi|)e^{-|\xi|}.
\end{equation*}

\begin{theorem}\label{th-cv-gamma}
  Under the assumptions of Theorem~\ref{thm-moments-decay}, if there exists some~$k_0\geqslant3$ such that~$\frac{1}{2^{2k_0 -1}}< \delta$ and~$\frac{M^0_{2k_0}}{M^0_2}$ is small enough, 
  then we have that, for some values of~$s \in (2,3)$, there exists a constant~$L>0$ such that 
  \begin{equation*}
   d_s(\gamma,\gamma_\infty)(t) \leqslant \big(d_s(\gamma_0,\gamma_\infty)+L\big)e^{-\lambda_s t},
  \end{equation*}
  with~$\lambda_s := 1-\frac{s}4-2^{1-s}$, which is positive for~$s \in (2,3)$. 
\end{theorem}

\begin{remark}\label{rk-cv-gamma} When~$\delta-\frac1{2^{2k_0-1}}>\frac14$, this estimation is valid for all~$s\in(2,3)$. Otherwise, there exist~$2<s_0<\bar{s}<3$ such that this estimation is valid on~$(2,s_0)$. For~$s=s_0$ we obtain~$d_{s_0}(\gamma,\gamma_\infty)(t) \leqslant \big(d_{s_0}(\gamma_0,\gamma_\infty)+Lt\big)e^{-\lambda_{s_0} t}$, and for~$s\in(s_0,\bar{s})$, we obtain that there exists a constant~$c_s\in(0,\lambda_s)$ such that
  \[d_{s}(\gamma,\gamma_\infty)(t) \leqslant d_{s}(\gamma_0,\gamma_\infty)e^{-\lambda_{s} t}+Le^{-c_{s} t}.\]
  In all these situations, the constant~$L$ depends on~$\delta,s$ and the initial moments~$M_{2k}^0$ for~$1\leqslant k\leqslant k_0$.
\end{remark}
To prove this last theorem, we prove a contraction property on the reproduction operator~$B_0$ in the distance~$d_s$,
which is due to the quadratic formulation of its Fourier transform,
and combine it with the results of Theorems~\ref{thm-stability-dirac} and~\ref{thm-moments-decay} that enable to control 
the remainder~$R$. Although the model under study is uncommon from a biologial viewpoint,
we hope that the method developed here, based on the contractive property in the Fourier distance
and the control on the selection effects, can be extended to other biological problems.

  \begin{remark}One interpretation of the result of Theorem~\ref{th-cv-gamma} regarding the original density~$g$ is that the self-similar profile~$\frac1{\sqrt{M_2(t)}}\gamma_\infty(\sqrt{M_2(t)}(x-\overline{x}(t)))$ is a better approximation of the solution than the Dirac mass located at~$\overline{x}(t)$ for large times~$t$. Indeed we may for instance use the distance~$d_2$ defined as in~\eqref{def-fourier-distance} for probability densities having same first moments. With the scaling properties of~$d_2$ and the fact that~$(d_s)^{\frac{2}s}$ controls~$d_2$ (see~\cite[Proposition 2.9]{Car-Tos-2007}), we indeed obtain
    \begin{align*} d_2\big(g(t,\cdot),\tfrac1{\sqrt{M_2(t)}}\gamma_\infty(\sqrt{M_2(t)}(x-\overline{x}(t)))\big)&=M_2(t)d_2(\gamma(t,\cdot),\gamma_\infty)\\
      &\lesssim M_2(t)\big(d_s(\gamma(t,\cdot),\gamma_\infty)\big)^{\frac2s}\lesssim  M_2(t)e^{-\frac{2\lambda_s}s t},
    \end{align*}
which means, since~$M_2(t)$ behaves as~$e^{-\frac{t}2}$, that the decay rate is at least~$\frac12+\frac{2\lambda_s}s$, while \begin{equation*}d_2(g(t,\cdot),\delta_{\overline{x}(t)})=M_2(t)d_2(\gamma(t,\cdot),\delta_{0})\geqslant M_2(t)\big(d_2(\gamma_\infty,\delta_0)-d_2(\gamma(t,\cdot),\gamma_\infty)\big),
\end{equation*}
therefore the decay rate is here at most~$\frac12$ (since~$d_2(\gamma(t,\cdot),\gamma_\infty)$ converges to~$0$).
  \end{remark}

\section{General results}\label{section-general-results}

\subsection{Global existence and uniqueness of measure solutions}\label{section-exist-uniq}

We want to give a meaning to our model~\eqref{model} for an arbitrary nonnegative measure of finite mass on~$\mathbb{R}$. If~$f$ is a smooth solution of~\eqref{model}, thinking of~$B_0[f(t,\cdot)](x)$ as a prescribed source term, we can interpret~\eqref{model} pointwise for every~$x\in\mathbb{R}$ as an ordinary differential equation in time, therefore for all~$t\geqslant0$ and for all~$x\in\mathbb{R}$ we have
\[f(t,x)=f^0(x)e^{-m(x)t}+\int_0^te^{-m(x)(t-s)}B_0[f(s,\cdot)](x)\d s,\]
and this will be the starting point of our definition of a measure solution. In the following, we will often use the abusive notation~$f(x)\d x$ even if~$f$ is only a measure (or~$f(t,x)\d x$ if the measure depends on time). Similarly, if~$h$ a measurable function in~$\mathbb{R}$, we may write~$\varphi(x)=h(x)f(x)$ to define a signed measure~$\varphi$, even if~$f$ is only a measure.

We denote by~$\mathcal{M}$ the space of finite signed (Borel) measures on~$\mathbb{R}$. Using the total variation norm defined by
\begin{equation*}
  \|\nu\|_{TV}=\sup \sum_{i=1}^n|\nu(A_i)|,
\end{equation*}
where the supremum is taken over all finite partitions of~$\mathbb{R}$ by Borel sets~$(A_i)_{1\leqslant i\leqslant n}$, this turns~$\mathcal{M}$ into a Banach space~\cite{cohn2013measure}.

Furthermore, we have that for any measurable function~$h$ taking values in~$[-H,H]$, we have
\begin{equation*}
  \Big|\int_\mathbb{R}h(x)\d \nu(x)\Big|\leqslant H\|\nu\|_{TV},
\end{equation*}
and therefore we obtain that~$\|\nu\|_{TV}=\sup|\int_\mathbb{R}h(x)\d \nu(x)|$ where the supremum is taken over all measurable function~$h$ taking values in~$[-1,1]$.

The set~$\mathcal{M}_+\subset\mathcal{M}$ of nonnegative measures with finite mass is a closed subspace of~$\mathcal{M}$.
Indeed, if~$\nu_n$ converges to~$\nu$ in total variation norm, 
then we have~$|\nu_n(A)-\nu(A)|\leqslant\|\nu_n-\nu\|_{TV}\to0$ for all Borel set~$A\subset\mathbb{R}$, 
thus~$\nu_n(A)\to\nu(A)$. So, if~$\nu_n$ is a nonnegative measure for all~$n$, then~$\nu$ is also nonnegative.

Notice that if~$\nu\in\mathcal{M}_+$, then~$B_0[\nu]$ also belongs to~$\mathcal{M}_+$, has same mass, and for all bounded (or nonnegative) and measurable~$h$ we have
\begin{equation}\label{weak-form-B0}
  \int_\mathbb{R}h(x)B_0[\nu](x)\d x=\frac{\int_{\mathbb{R}\times\mathbb{R}}h\big(\frac{y+z}2\big)\nu(y)\d y\,\nu(z) \d z}{\int_\mathbb{R}\nu(x)\d x},
\end{equation}
when~$\nu$ has positive mass (thus the conservation of mass is obtained with~$h=1$), and~$B_0[\nu]=0$ if~$\nu$ is the zero measure. 
\bigskip

We are therefore ready to study the properties of measure solutions of our model~\eqref{model}, given by Definition~\ref{def-weak-solution}. Notice that all terms in~\eqref{eq-weak-solution} are well defined as nonnegative integrals, and if~$m$ is measurable and bounded below, they are finite, since~$x\mapsto e^{-m(x)t}$ is measurable and bounded for~$t\geqslant0$. Consequently, if~$f$ is a weak solution and~$h$ is a measurable bounded (resp. nonnegative) function, being a uniform (resp. monotone) limit of a sequence of measurable simple functions, we have :
  \begin{equation}\label{weak-form-h}
    \int_\mathbb{R}h(x)f(t,x)\d x=\int_\mathbb{R}h(x)e^{-m(x)t}f^0(x)\d x + \int_0^t\int_\mathbb{R}h(x)e^{-m(x)(t-s)}B_0[f(s,\cdot)](x)\d x\, \d s.
  \end{equation}

  Thanks to this property, we can have a useful semigroup property for the solutions of our model.
  \begin{proposition}\label{prop-semigroup}
    Let~$t_0>0$. If~$f$ is a weak solution to the model~\eqref{model}, then~$t\mapsto f(t_0+t,\cdot)$ is also a solution for the initial condition~$f(t_0,\cdot)$.
  \end{proposition}
  \begin{proof}
    We fix~$t>0$. If~$A$ is a Borel set, we have
    \begin{align*}
      \int_Af(t_0+t,x)\d x&=\int_Ae^{-m(x)(t_0+t)}f^0(x)\d x + \int_0^{t_0+t}\int_Ae^{-m(x)(t_0+t-s)}B_0[f(s,\cdot)](x)\d x\, \d s.\\
                       &= \int_Ae^{-m(x)(t_0+t)}f^0(x)\d x +\int_0^{t_0}\int_Ae^{-m(x)(t_0+t-s)}B_0[f(s,\cdot)](x)\d x\, \d s\\
                         &\hspace{2cm} + \int_0^{t}\int_Ae^{-m(x)(t-s)}B_0[f(t_0+s,\cdot)](x)\d x\, \d s.
    \end{align*}
    Since the function~$\mathbf{1}_{A}(x)e^{-m(x)t}$ is measurable (and nonnegative), we have thanks to~\eqref{weak-form-h} that the first two terms of the right-hand side of this last equality combine and we get
    \begin{equation*}
      \int_Af(t_0+t,x)\d x = \int_Ae^{-m(x)t}f(t_0,x)\d x+ \int_0^{t}\int_Ae^{-m(x)(t-s)}B_0[f(t_0+s,\cdot)](x)\d x\, \d s,
    \end{equation*}
    and this ends the proof.
  \end{proof}

We proceed to the proof of Theorem~\ref{thm-existence-uniqueness}, which provides global existence and uniqueness of solutions to~\eqref{model} in $\mathcal{M}_+$.

\begin{proof} First of all, we need to study some contraction properties 
  of~$B_0$. If~$h$ is a measurable function from~$\mathbb{R}$ to~$[-H,H]$, and~$\nu,\widetilde{\nu}\in\mathcal{M}_+$, we write~$\rho=\int_\mathbb{R}\nu(x)\d x$ and~$\widetilde{\rho}=\int_\mathbb{R}\widetilde{\nu}(x)\d x$ 
  and we have, when~$\rho,\widetilde{\rho}>0$,
{\allowdisplaybreaks  \begin{align*}
    \Big|\int_\mathbb{R}h(x)B_0[\nu](x)&\d x -\int_\mathbb{R}h(x)B_0[\widetilde{\nu}](x)\d x\Big|\\
    &=\Big|\frac{\int_{\mathbb{R}\times\mathbb{R}}h\big(\frac{y+z}2\big)\nu(y)\d y \,\nu(z) \d z}{\rho}-\frac{\int_{\mathbb{R}\times\mathbb{R}}h\big(\frac{y+z}2\big)\widetilde{\nu}(y)\d y \,\widetilde{\nu}(z)\d z}{\widetilde{\rho}}\Big|\\
\begin{split}         &\leqslant \Big|\frac{\int_{\mathbb{R}\times\mathbb{R}}h\big(\frac{y+z}2\big)\nu(y)\d y \,\nu(z)\d z}{\rho}-\frac{\int_{\mathbb{R}\times\mathbb{R}}h\big(\frac{y+z}2\big)\nu(y)\d y \,\widetilde{\nu}(z) \d z}{\rho}\Big|\\
         &\hspace{1cm}+\Big|\frac{\int_{\mathbb{R}\times\mathbb{R}}h\big(\frac{y+z}2\big)\nu(y)\d y\,\widetilde{\nu}(z) \d z}{\rho}-\frac{\int_{\mathbb{R}\times\mathbb{R}}h\big(\frac{y+z}2\big)\nu(y)\d y\, \widetilde{\nu}(z) \d z}{\widetilde{\rho}}\Big|\\
                         &\hspace{1cm}+\Big|\frac{\int_{\mathbb{R}\times\mathbb{R}}h\big(\frac{y+z}2\big)\nu(y)\d y\,\widetilde{\nu}(z)\d z}{\widetilde{\rho}}-\frac{\int_{\mathbb{R}\times\mathbb{R}}h\big(\frac{y+z}2\big)\widetilde{\nu}(y)\d y\,\widetilde{\nu}(z)\d z}{\widetilde{\rho}}\Big|
                       \end{split}\\
                        \begin{split}
                         &\leqslant \frac{\int_{\mathbb{R}}\big|\int_\mathbb{R}h\big(\frac{y+z}2\big)\nu(z)\d z-\int_\mathbb{R}h\big(\frac{y+z}2\big)\widetilde{\nu}(z)\d z\big|\nu(y)\d y }{\rho}\\
         &\hspace{1cm}+\Big|\frac{1}{\rho}-\frac{1}{\widetilde{\rho}}\Big|\int_{\mathbb{R}\times\mathbb{R}}\big|h\big(\frac{y+z}2\big)\big|\nu(y)\d y\,\widetilde{\nu}(z) \d z\\
                         &\hspace{1cm}+\frac{\int_{\mathbb{R}}\big|\int_\mathbb{R}h\big(\frac{y+z}2\big)\nu(y)\d y-\int_\mathbb{R}h\big(\frac{y+z}2\big)\widetilde{\nu}(y)\d y\big|\,\widetilde{\nu}(z)\d z}{\widetilde{\rho}}
                       \end{split}
                        \\
    &\leqslant2H \|\nu-\widetilde{\nu}\|_{TV}+H|\widetilde{\rho}-\rho|.
  \end{align*}}
  We actually have~$\rho=\|\nu\|_{TV}$ (since~$\nu$ is nonnegative) and~$\widetilde{\rho}=\|\widetilde{\nu}\|_{TV}$, therefore we obtain
  \begin{equation}\label{hB0-TV}
    \Big|\int_\mathbb{R}h(x)B_0[\nu](x)\d x -\int_\mathbb{R}h(x)B_0[\widetilde{\nu}](x)\d x\Big|\leqslant3H\|\nu-\widetilde{\nu}\|_{TV},
  \end{equation}
  and one easily checks that this inequality is still valid if~$\rho=0$ or~$\widetilde{\rho}=0$.
  
\bigskip

  We now fix~$f^0\in\mathcal{M}_+$ and~$T>0$. For any fixed~$L>0$ (to be determined later on), the space~$\mathcal{X}=C([0,T],\mathcal{M})$ with the norm
  \begin{equation*}
    \|f\|_{\mathcal{X}}=\sup_{t\in[0,T]}e^{-Lt}\|f(t,\cdot)\|_{TV},
  \end{equation*}
  is a Banach space. The subset~$\mathcal{X}_+=C([0,T],\mathcal{M}_+)$ of nonnegative finite measures varying continuously in time is a closed subset of~$\mathcal{X}$, and we will construct a fixed point by Picard iteration on this complete metric space.
  
  For~$f\in C([0,T],\mathcal{M}_+)$, we denote~$\Phi(f)$ the time-dependent measure defined for all Borel set~$A\subset\mathbb{R}$ by
  \begin{equation}\label{def-Phi}
    \int_A\Phi(f)(t,x)\d x=\int_Ae^{-m(x)t}f^0(x)\d x + \int_0^t\int_Ae^{-m(x)(t-s)}B_0[f(s,\cdot)](x)\d x\, \d s,
  \end{equation}
  and we are therefore looking for a fixed point in~$\mathcal{X}_+$ of the operator~$\Phi$. Let us first check that~$\Phi$ is well-defined from~$\mathcal{X}_+$ to itself.

 If~$h:\mathbb{R}\to[-1,1]$ is a measurable function, being a uniform limit of a sequence of measurable simple functions, we have
  \begin{equation*} 
    \int_\mathbb{R}h(x)\Phi(f)(t,x)\d x=\int_\mathbb{R}h(x)e^{-m(x)t}f^0(x)\d x + \int_0^t\int_\mathbb{R}h(x)e^{-m(x)(t-s)}B_0[f(s,\cdot)](x)\d x\, \d s.
  \end{equation*}
  Then, taking $K \geqslant 0$ such that $m$ is bounded below by $-K$, we obtain, for $t,t' \geqslant 0$,
  \begin{align*}
    \Big|\int_\mathbb{R}&h(x)\Phi(f)(t',x)\d x-\int_{\mathbb{R}}h(x)\Phi(f)(t,x)\d x\Big|\\
                    &\leqslant\int_\mathbb{R}|e^{-m(x)t'}-e^{-m(x)t}|f^0(x)\d x +\int_{t}^{t'}\int_{\mathbb{R}}|e^{-m(x)(t'-s)}|B_0[f(s,\cdot)](x)\d x\, \d s \\
                    &\hspace{1cm}+ \int_0^t\int_\mathbb{R}|e^{-m(x)(t'-s)}-e^{-m(x)(t-s)}|B_0[f(s,\cdot)](x)\d x\, \d s \\
            &\leqslant Ke^{KT}|t'-t|\|f^0\|_{TV}+\int_0^tKe^{KT}|t'-t|\|f(s,\cdot)\|_{TV}\d s+\int_{t}^{t'}e^{KT}\|f(s,\cdot)\|_{TV} \d s,
  \end{align*}
  where we used the fact that~$e^{-m(x)t}$ is~$Ke^{KT}$-Lipschitz in time for all~$x\in\mathbb{R}$, and the fact that the total mass of~$B_0[f(s,\cdot)]$ (which is its total variation since~$f$ is nonnegative) is the same as the one of~$f(s,\cdot)$. Furthermore we have~$\|f(s,\cdot)\|_{TV}\leqslant e^{Ls}\|f\|_{\mathcal{X}}\leqslant e^{LT}\|f\|_{\mathcal{X}}$ for all~$s\in[0,T]$. Therefore we obtain
  \[\|\Phi(f)(t',\cdot)-\Phi(f)(t,\cdot)\|_{TV}\leqslant(Ke^{KT}\|f^0\|_{TV}+(KT+1)e^{(K+L)T}\|f\|_{\mathcal{X}})|t'-t|,\]
  showing that~$t\mapsto\Phi(f)(t,\cdot)$ is continuous in time (and actually Lipschitz). Notice that since all the measures in the definition of~$\Phi$ are nonnegative, we get that~$\Phi(f)(t,\cdot)$ is a nonnegative measure for all time, and the operator~$\Phi$ therefore sends~$\mathcal{X}_+$ into itself.

  For~$f,\widetilde{f}\in\mathcal{X}_+$, we fix a measurable function~$h:\mathbb{R}\to[-1,1]$, and we compute similarly, thanks to~\eqref{hB0-TV}
  \begin{align*}
    \Big|\int_\mathbb{R}&h(x)\Phi(f)(t,x)\d x-\int_{\mathbb{R}}h(x)\Phi(\widetilde{f})(t,x)\d x\Big|\\
                    &\leqslant\int_{0}^{t}\Big|\int_\mathbb{R}h(x)e^{-m(x)(t-s)}B_0[f(s,\cdot)](x)\d x-\int_\mathbb{R}h(x)e^{-m(x)(t-s)}B_0[\widetilde{f}(s,\cdot)](x)\d x\Big|\, \d s\\
            &\leqslant\int_0^t3e^{KT}\|f(s,\cdot)-\widetilde{f}(s,\cdot)\|_{TV}\, \d s\\
    &\leqslant\int_0^t3e^{KT}\|f-\widetilde{f}\|_{\mathcal{X}}\,e^{Ls} \,\d s=\frac{3e^{KT}}{L}(e^{Lt}-1)\|f-\widetilde{f}\|_{\mathcal{X}},
  \end{align*}
  which shows that
  \[\|\Phi(f)(t,\cdot)-\Phi(\widetilde{f})(t,\cdot)\|_{TV}\leqslant\frac{3e^{KT}}{L}e^{Lt}\|f-\widetilde{f}\|_{\mathcal{X}},\]
  and therefore
  \[\|\Phi(f)-\Phi(\widetilde{f})\|_{\mathcal{X}}\leqslant\frac{3e^{KT}}{L}\|f-\widetilde{f}\|_{\mathcal{X}}.\]
  So, for~$L>3e^{KT}$,~$\Phi$ is a contraction and therefore has a unique fixed point in $\mathcal{X}_+$.  
\end{proof}

\subsection{Evolution of moments}

In order to investigate the large time behavior of a solution~$f$, we will need to control 
the moments. The key observation is that any bound on finite moments is preserved by~$B_0$.

\begin{lemma}\label{moment-B0}
  If~$\nu\in\mathcal{M_+}$, then for all~$k\in\mathbb{N}$, we have 
  \begin{equation}\label{estimate:moment_B0}
  \int_\mathbb{R}|x|^kB_0[\nu](x)\d x\leqslant\int_\mathbb{R}|x|^k\nu(x)\d x.
  \end{equation}

  Furthermore, if~$\int_\mathbb{R}|x|^k\nu(x)\d x$ is finite and~$\nu$ has positive mass, all the moments up to order~$k$ are well defined and we have
  \begin{equation}\label{expansion-moment-B0}
    \mu_k(B_0[\nu]) =\sum_{j=0}^k\binom{k}{j}\frac{\mu_j(\nu)\mu_{k-j}(\nu)}{2^k\mu_0(\nu)}.
  \end{equation}
\end{lemma}
\begin{proof}
Let us remark that both sides of the first inequality are well defined as nonnegative integrals. We suppose that~$\int_\mathbb{R}|x|^k\nu(x)\d x$ 
is finite and that~$\nu$ has positive mass,  thanks to~\eqref{weak-form-B0} we get :
  \begin{align*}
    \int_\mathbb{R}|x|^kB_0[\nu](x)\d x &= \frac{\int_{\mathbb{R}\times\mathbb{R}}\frac{|y+z|^k}{2^k}\nu(y)\d y\, \nu(z)\d z}{\int_\mathbb{R}\nu(x)\d x}\\
                          &\leqslant\frac{\sum_{j=0}^k\binom{k}{j}\int_\mathbb{R}|y|^j\nu(y)\d y\int_\mathbb{R}|z|^{k-j}\nu(z)\d z}{2^k\int_\mathbb{R}\nu(x)\d x}.
  \end{align*}
  Then, by Hölder inequality, for~$0\leqslant j\leqslant k$ we obtain
  \[\int_\mathbb{R}|y|^j\nu(y)\d y\leqslant\Big(\int_\mathbb{R}\nu(x)\d x\Big)^{1-\frac{j}k}\Big(\int_\mathbb{R}|x|^k\nu(x)\d x\Big)^{\frac{j}k},\]
and we therefore get
  \[\int_\mathbb{R}|x|^kB_0[\nu](x)\d x\leqslant\frac{\sum_{j=0}^k\binom{k}{j}\big(\int_\mathbb{R}\nu(x)\d x\big)\big(\int_\mathbb{R}|x|^k\nu(x)\d x\big)}{2^k\int_\mathbb{R}\nu(x)\d x}=\int_\mathbb{R}|x|^k\nu(x)\d x.\]
  
  Now, by using the dominated convergence theorem, we can approximate~$x^k$ by a bounded measurable function~$h_N(x)$, for instance by setting
  \begin{equation}\label{def-hN}
    h_N(x)=\begin{cases}(-N)^k&\text{if }x\leqslant-N\\
      x^k&\text{if }-N\leqslant x\leqslant N\\
      N^k&\text{if }x\geqslant N,
    \end{cases}
  \end{equation}
  and obtain, by applying~\eqref{weak-form-B0} to~$h_N$, that in the limit~$N\to\infty$, we have
  \begin{align*}
    \int_\mathbb{R}x^kB_0[\nu](x)\d x & = \frac{\int_{\mathbb{R}\times\mathbb{R}}\frac{(y+z)^k}{2^k}\nu(y)\d y \,\nu(z)\d z}{\int_\mathbb{R}\nu(x)\d x},\\
                           &= \frac{\sum_{j=0}^k\binom{k}{j}\int_\mathbb{R}y^j\nu(y)\d y\int_\mathbb{R}z^{k-j}\nu(z)\d z}{2^k\int_\mathbb{R}\nu(x)\d x}.
 \end{align*}
  which is then well-defined and provides the formula~\eqref{expansion-moment-B0}.
\end{proof}

Thanks to this property, we can now give the proof of the following proposition, which shows that finiteness of moments is propagated in time.

\begin{proposition}\label{moment-prop}
  Let~$f^0\in\mathcal{M}_+$,~$m$ be measurable and bounded below by~$-K\leqslant0$, and~$f$ be the global solution to the model~\eqref{model} given by Theorem~\ref{thm-existence-uniqueness}.
If for some~$k\in\mathbb{N}$,~$\int_\mathbb{R}|x|^kf^0(x)\d x<+\infty$, then for all~$t\geqslant0$, we have
  \[\int_{\mathbb{R}}|x|^kf(t,x)\d x \leqslant e^{(K+1)t}\int_\mathbb{R}|x|^kf^0(x)\d x.\]
\end{proposition}

\begin{proof}
  Back to the Picard iteration sequence~$(f_n)_n$ with~$f_{n+1} = \Phi(f_n)$ 
  and~$f_0 = f^0$, where~$\Phi$ is defined in~\eqref{def-Phi} in the proof of Theorem~\ref{thm-existence-uniqueness}, 
  we proceed by induction and assume as the induction hypothesis for~$n\geqslant0$ that 
  \begin{equation*}
    \int_{\mathbb{R}}|x|^kf_n(t,x)\d x \leqslant e^{(K+1)t}\int_\mathbb{R}|x|^kf^0(x)\d x.
  \end{equation*}
Then we have, since~$h_N (x)$ defined in~\eqref{def-hN} is bounded and measurable,
  \begin{align*}
    \int_\R |h_N (x)| f_{n+1}(t,x)\d x &= \int_\R e^{-m(x)t} |h_N (x)| f^0(x)\d x \\
    &\hspace{2cm}+    \int_0^t \int_\R e^{-m(x)(t-s)} |h_N (x)| B_0[f_n(s, \cdot)](x) \d x \d s\\
    &\leqslant e^{Kt} \int_\R |x|^k f^0(x)\d x + \int_0^t e^{K(t-s)} \int_\R |h_N (x)| B_0[f_n(s, \cdot)](x) \d x \d s.
  \end{align*}
By~\eqref{estimate:moment_B0} in Lemma~\ref{moment-B0} and the induction hypothesis, we obtain
  \begin{align*}
    \int_\R |h_N| (x) f_{n+1}(t,x)\d x&\leqslant  \left( e^{Kt}+ \int_0^t e^{K(t-s)}e^{(K+1)s}\d s\right)\int_\R |x|^k f^0(x)\d x \\
    &= \left( 1+ \int_0^t e^{s}\d s\right)e^{Kt}\int_\R |x|^k f^0(x)\d x = e^{(K+1)t}\int_\R |x|^k f^0(x)\d x.
  \end{align*}
  By construction of~$f$, passing to the limit~$n \to \infty$ we obtain 
  \begin{equation*}
    \int_\R |h_N (x)| f(t,x)\d x \leqslant e^{(K+1)t}\int_\R |x|^k f^0(x)\d x.
  \end{equation*}
  By monotone convergence theorem, we can pass to the limit~$N \to \infty$ to get 
  \begin{equation*}
    \int_\R |x|^k f(t,x)\d x \leqslant e^{(K+1)t}\int_\R |x|^k f^0(x)\d x.\qedhere
  \end{equation*}
\end{proof}

Thanks to these bounds, we will obtain more precise properties on the evolution 
of the (even) moments of the solution. In some cases, we do not necessarily 
have that the moments are differentiable in time, but we obtain estimates similar 
to differential inequalities. A good framework to work with continuous functions 
of time is set in the following lemma.
\begin{lemma}\label{lemma-dplus}
  Given~$t\mapsto y(t)$ a function of time, we write
  \[\frac{\d^+}{\d t}y(t)=\limsup_{\tau\to0,\tau>0}\frac{y(t+\tau)-y(t)}{\tau}.\]
  Then, for any~$T>0$, if~$y$ is continuous and satisfies~$\frac{\d^+}{\d t}y(t)\leqslant0$ on~$[0,T]$, then it is nonincreasing on~$[0,T]$.

  Consequently, if, for some integrable functions~$a$ and~$b$, we have a continuous function~$y$ satisfying the following Grönwall-like estimate for all 
time~$t\in[0,T]$: 
\[\frac{\d^+}{\d t}y(t)\leqslant a(t)y(t)+b(t),\]
then we have for all~$t\in[0,T]$,
\[y(t)\leqslant y(0)e^{\int_0^ta(s)\d s}+\int_0^tb(u)e^{\int_u^ta(s)\d s}\d u.\]
\end{lemma}

\begin{proof}
  The proof (by contraposition) is based on a result similar to 
  the mean value theorem : if there exist some times~$t_0<t_1$ such 
  that~$y$ is continuous on~$[t_0,t_1]$ and~$y(t_0)<y(t_1)$, 
  then there exists~$t_*\in[t_0,t_1)$ such that~$\liminf\limits_{\tau\to0,\tau>0}\frac{y(t_*+\tau)-y(t_*)}{\tau}\geqslant\frac{y(t_1)-y(t_0)}{t_1-t_0}$.
  Thus, we can conclude that~$\frac{\d^+}{\d t}y(t_*)\geqslant\frac{y(t_1)-y(t_0)}{t_1-t_0}>0$. Surprisingly, we see that we only need the liminf of the slope to be nonpositive to show that a continuous function is nonincreasing (the proof also works for lower-semicontinuous functions). This result may be classical, but we did not find any reference, this is why we include the proof here.
  
  To show the claimed result, if there exist such times~$t_0$ and~$t_1$, we introduce
  \[z(s)=y((1-s)t_0+st_1)-(1-s)y(t_0)-sy(t_1).\]
  The function~$z$ is continuous on~$[0,1]$, and 
  satisfies~$z(0)=z(1)=0$. Therefore it reaches its minimum 
  on~$s_*\in[0,1]$, and without loss of generality we may 
  assume~$s_*\neq1$ (since in that case the minimum is also 
  reached at~$s=0$). We therefore get~$\liminf_{s\to s_*,s>s_*}\frac{z(s)-z(s_*)}{s-s_*}\geqslant0$, 
  which corresponds to the claim, with~$t_*=(1-s_*)t_0+s_*t_1$.

  Finally, the resolution of the Grönwall-like estimate is a direct consequence of the first part of the lemma, by setting~$z(t)=y(t)e^{-\int_0^ta(s)\d s}-\int_0^tb(u)e^{-\int_0^ua(s)\d s}\d u$. Then the continuous function~$z$ is such that~$\frac{\d^+}{\d t}z\leqslant0$, therefore nonincreasing.  
\end{proof}

We are now ready to provide this kind of estimates for the even 
moments of the solution, which will be useful when we derive estimates 
for the moments of the centered and normalized distribution.
We proceed to the proof of the following proposition.
\begin{proposition}\label{prop-continuity-moments-bounds}
  Let~$f^0\in\mathcal{M}_+$,~$m$ be measurable and bounded below by~$-K\leqslant0$, and~$f$ be the global solution to the model~\eqref{model} given by Theorem~\ref{thm-existence-uniqueness}.

  For~$k\geqslant0$, if~$\int_\mathbb{R}|x|^kf^0\d x<+\infty$, then~$t\mapsto\mu_k(f(t,\cdot))$ is continuous. Furthermore, if~$k$ is even, for all~$t\geqslant0$, we have 
  \begin{equation}
    \frac{\d^+}{\d t}\mu_k(f(t,\cdot))\leqslant\mu_k(B_0[f(t,\cdot)])+K\mu_k(f(t,\cdot)).\label{eq-subsolution-moments-f}
  \end{equation}
\end{proposition}

\begin{proof}
  We fix~$T>0$. For~$0\leqslant t<t+\tau<T$, and~$h_N$ as in~\eqref{def-hN}, we write, thanks to~\eqref{weak-form-h} and to Proposition~\ref{prop-semigroup} :
\begin{align*}
    \int_\R h_N (x) f(t+\tau,x)\d x &= \int_\R e^{-m(x)\tau} h_N (x) f(t,x)\d x \\
    &\hspace{2cm}+    \int_0^{\tau} \int_\R e^{-m(x)(\tau-s)} h_N (x) B_0[f(t+s, \cdot)](x) \d x \d s.
  \end{align*}
  By dominated convergence as~$N\to+\infty$, thanks to Lemma~\ref{moment-B0}, Proposition~\ref{moment-prop} and the fact that~$m(x)\geqslant-K$, we obtain
 \begin{equation}\label{eq-muk-weak1}
    \mu_k(f(t+\tau,\cdot)) = \int_\R e^{-m(x)\tau} x^k f(t,x)\d x  + \int_0^{\tau} \int_\R e^{-m(x)(\tau-s)} x^k B_0[f(t+s, \cdot)](x) \d x \d s.
  \end{equation}

  In particular, by setting~$t=0$ and writing~$t$ instead of~$\tau$, we get that for all~$t\geqslant0$, 
 \begin{equation}\label{eq-muk-weakt0}
    \mu_k(f(t,\cdot)) = \int_\R e^{-m(x)t} x^k f^0(x)\d x + \int_0^{t} \int_\R e^{-m(x)(t-s)} x^k B_0[f(s, \cdot)](x) \d x \d s.
  \end{equation}

  We write this expression under the form~$\mu(f(t,\cdot)) = y(t) + \int_0^Tz(t,s)\d s$, by setting
  \begin{align*}
    y(t)&=\int_\R e^{-m(x)t} x^k f^0(x)\d x,\\
    z(t,s)&=\begin{cases}\int_\R e^{-m(x)(t-s)} x^k B_0[f(s, \cdot)](x) \d x& \text{if }s<t\\0 & \text{if }s\geqslant t
    \end{cases}.
  \end{align*}

  By dominated convergence, since~$e^{-m(x)t} |x|^k \leqslant e^{KT}|x|^k$ which is~$f^0$ integrable, we obtain that~$y$ is continuous on~$[0,T]$. Similarly~$t\mapsto z(t,s)$ is continuous on~$(s,T]$ (thanks to Proposition~\ref{moment-prop}, \eqref{estimate:moment_B0} in Lemma~\ref{moment-B0} and the fact that~$m(x)\geqslant-K$) and we also have
  \[|z(t,s)|\leqslant e^{K(t-s)}\int_\mathbb{R}|x|^kf(s,x)\d x\leqslant e^{K(t-s)+(K+1)s}\int_\mathbb{R}|x|^kf^0(x)\d x.\]
  Consequently,~$z(t,s)$ is uniformly bounded on~$[0,T]^2$ by some~$C>0$.

  Obviously,~$t\mapsto z(t,s)$ is also continuous on~$[0,s)$, since it is equal to zero.
  We now want to prove that~$t\mapsto\int_0^Tz(t,s)\d s$ is continuous. 
  We fix~$t_0\in[0,T]$, and we have that for all~$s\neq t_0$,~$z(t,s)$ converges to~$z(t_0,s)$ as~$t\to t_0$, by the continuity properties of~$z$. So by dominated convergence, we get that~$\int_0^Tz(t,s)\d s$ converges to~$\int_0^Tz(t_0,s)\d s$ as~$t\to t_0$. This ends the proof of the time continuity of~$\mu_k(f(t,\cdot))$.

  We now suppose that~$k$ is even. From~\eqref{eq-muk-weak1}, we obtain
  \begin{align*}
    \mu_k(f(t+\tau,\cdot))&\leqslant\mu_k(f(t,\cdot))e^{K\tau}+\int_0^{\tau} e^{K(\tau-s)}\mu_k(B_0[f(t+s,\cdot)])\d s\nonumber \\
    &\leqslant e^{K\tau}\Big(\mu_k(f(t,\cdot)) + \int_0^{\tau} \mu_k(B_0[f(t+s,\cdot)])\d s\Big). 
  \end{align*}
 We therefore obtain
  \begin{equation}\label{variation-muk}\frac{\mu_k(f(t+\tau,\cdot))-\mu_k(f(t,\cdot))}{\tau}\leqslant\frac{e^{K\tau}-1}{\tau}\mu_k(f(t,\cdot))+\frac{e^{K\tau}}{\tau}\int_0^{\tau} \mu_k(B_0[f(t+s,\cdot)])\d s.
  \end{equation}
  By the formula~\eqref{expansion-moment-B0} and the fact that all moments of~$f$ of order~$j\leqslant k$ are continuous in time, we get that~$t\mapsto\mu_k(B_0[f(t,\cdot)])$ is continuous, and therefore the right-hand side of~\eqref{variation-muk} converges to~$K\mu_k(f(t,\cdot))+\mu_k(B_0[f(t,\cdot)])$ as~$\tau\to0$, and this proves~\eqref{eq-subsolution-moments-f}.
\end{proof}

\subsection{Centered moments of the normalized population density}

To obtain estimates on the centered moments of~$f$, the inequalities provided by 
Proposition~\ref{prop-continuity-moments-bounds} will not be sufficient ; we will 
first need a more precise description of the evolution of the moments~$\mu_k$ 
(in particular we will have to compute the time derivatives of~$\mu_0$ and~$\mu_1$). 
To this aim, we make the stronger assumption~\ref{assumption-growth} on the growth of~$m$.

Under this assumption, if~$f^0$ has finite second moment, then the evolution of quantities of the form~$\int h(x)f(t,x)\d x$ (when~$h$ is measurable and bounded) provides a simpler weak formulation of~\eqref{model}, as stated in the next proposition.

\begin{proposition} \label{prop-weak-formulation2-f} Under assumption~\ref{assumption-growth}, we suppose that~$f^0$ has finite second moment and denote by~$f$ the global solution to the model~\eqref{model} given by Theorem~\ref{thm-existence-uniqueness}. Then for all measurable and bounded function~$h$, the quantity~$\int_\mathbb{R}h(x)f(t,x)\d x$ is differentiable in time and we have
\begin{equation}\label{weak-formula2}
    \ddt\Big(\int_\mathbb{R}h(x)f(t,x)\d x\Big)=\int_\R h(x)\big(B_0[f(t,\cdot)](x) -m(x)f(t,x)\big)\d x.
\end{equation}
\end{proposition}
\begin{proof}
  From Proposition~\ref{moment-prop} and Assumption~\ref{assumption-growth}, for all measurable and bounded function~$h$, 
   we want to obtain dominations of the time derivatives of the integrands appearing in~\eqref{weak-form-h} on~$[0,T]$. 
   First of all, we have
   \[ |h(x)m(x)e^{-m(x)t}|\leqslant Ce^{KT}(1+x^2)\sup_{\mathbb{R}}|h|,\]
   and the right-hand side is~$f^0$-integrable, therefore by dominated convergence theorem we get that the term~$\int_\mathbb{R}h(x)e^{-m(x)t}f^0(x)\d x$ is differentiable in time. We also have for all~$s\in[0,T]$, and for all~$t\geqslant s$,
   \[|h(x)m(x)e^{-m(x)(t-s)}|\leqslant Ce^{KT}(1+x^2)\sup_{\mathbb{R}}|h|,\]
   which is~$B_0[f(s,\cdot)]$-integrable for all~$s\in[0,T]$, thanks 
   to Lemma~\ref{moment-B0} and Proposition~\ref{moment-prop}. 
   Therefore by dominated convergence the derivative of~$\int_\mathbb{R}h(x)e^{-m(x)(t-s)}B_0[f(s,\cdot)](x)\d x$ 
   with respect to variable~$t$ is well defined, and uniformly bounded, on~$[s,T]$. By Leibniz rule on the formula~\eqref{weak-form-h}, the derivative of~$\int_\R h(x)f(t,x)\d x$ is well defined for all~$t>0$ and we have
\begin{align*}
  \ddt \left( \int_\R h(x)f(t,x)\d x\right) = &- \int_\R h(x) m(x)e^{-m(x)t}f^0(x)\d x
  + \int_\R h(x) B_0[f(t,\cdot)](x)\d x \\
   &- \int_0^t \int_\R h(x)m(x) e^{-m(x)(t-s)}B_0[f(s,\cdot)](x)\d x \d s.
\end{align*}
Once more, by dominated convergence theorem, we can 
approximate~$h(x)m(x)$ by a bounded measurable function~$\widetilde{h}_N(x)$ and 
obtain, applying~\eqref{weak-form-h} to~$\widetilde{h}_N$, that in the limit~$N\to\infty$, 
we have
\[\begin{split}\int_\R h(x) m(x)f(s,x) \d x=\int_\mathbb{R}h(x)&m(x)e^{-m(x)t}f^0(x)\d x \\
    &+ \int_0^t\int_\mathbb{R}h(x)m(x)e^{-m(x)(t-s)}B_0[f(s,\cdot)](x)\d x\, \d s.
  \end{split}\]
We therefore obtain
\begin{equation*}
  \ddt \left( \int_\R h(x)f(t,x)\d x\right) = \int_\R h(x) B_0[f(s,\cdot)](x)\d x
  - \int_\R h(x) m(x)f(s,x) \d x.
\end{equation*}
Hence, the second weak formulation~\eqref{weak-formula2}.
\end{proof}

\begin{remark} We chose in Assumption~\ref{assumption-growth} to have at most quadratic growth of the mortality rate~$m$. If instead we suppose that it does not grow faster that~$|x|^k$, the same reasoning would apply, and Proposition~\ref{prop-weak-formulation2-f} would still apply, replacing the requirement on~$f^0$ by asking~$\int_{\R}|x|^kf^0(x)\d x$ to be finite.
\end{remark}

We are now ready to give the evolution equations for the moments of the solution~$f$, which we state in the 
following proposition.
\begin{proposition}\label{prop-evolution-moments-f} Let~$k\in\mathbb{N}$. Under Assumption~\ref{assumption-growth}, if~$f^0$ has a moment of order~$k+2$ and~$f$ is the unique weak solution to the model~\eqref{model}, then~$\mu_k(f)$ is continuously differentiable in time and we have for all~$t\geqslant0$
  \begin{equation}\label{eq-evolution-moments-f}
    \ddt \mu_k(f(t,\cdot))=\sum_{j=0}^k\binom{k}{j}\frac{\mu_j(f(t,\cdot))\mu_{k-j}(f(t,\cdot))}{2^k\mu_0(f(t,\cdot))}-\int_\mathbb{R}m(x)x^kf(t,x)\d x.
  \end{equation}
\end{proposition}
\begin{proof}
Let us remark that by Definition~\ref{def-weak-solution} we obtain 
that~$\mu_0(f(t,\cdot))>0$ for any time. If not, we would 
have~$\int_\mathbb{R}e^{-m(x)t}f^0(x)\d x=0$ and therefore~$e^{-m(x)t}$ 
would be zero almost everywhere with respect to the measure~$f^0$, 
giving a contradiction.
By Proposition~\ref{prop-weak-formulation2-f}, using again the 
function~$h_N$ defined in~\eqref{def-hN}, we obtain 
\begin{equation*}
  \begin{split} \int_\mathbb{R}h_N(x)f(t,x)\d x = \int_\mathbb{R}&h_N(x)f^0(x)\d x \\
    &+ \int_0^t \Big(\int_\mathbb{R}h_N(x)B_0[f(s, x)] \d x- \int_\mathbb{R}h_N(x)m(x)f(s,x)\d x\Big) \d s.
  \end{split}
\end{equation*}
Thanks to Proposition~\ref{moment-prop} and to Assumption~\ref{assumption-growth},~$m(\cdot)f(s,\cdot)$ 
has a moment of order~$k$, uniformly on~$[0,t]$, therefore we obtain 
by dominated convergence theorem 
\begin{equation*}
  \mu_k(f(t,\cdot)) = \mu_k(f^0) + \int_0^t \big(\mu_k (B_0[f(s, \cdot)]) - \mu_k(m(\cdot)f(s,\cdot)) \big) \d s.
\end{equation*}
Then, as in the proof of Proposition~\eqref{prop-continuity-moments-bounds} for the time continuity of~$\mu_k(f(t,\cdot))$, we get similarly as in~\eqref{eq-muk-weakt0} that
 \[    \mu_k(m(\cdot)f(t,\cdot)) = \int_\R e^{-m(x)t} m(x) x^k  f^0(x)\d x + \int_0^{t} \int_\R e^{-m(x)(t-s)} m(x) x^k B_0[f(s, \cdot)](x) \d x \d s. \]
And we proceed in the same way, by dominated convergence thanks to 
Assumption~\ref{assumption-growth}, to prove that~$\mu_k(m(\cdot)f(t,\cdot))$ 
is continuous in time. Thanks to Propositions~\ref{moment-prop} 
and~\eqref{prop-continuity-moments-bounds}, we have the time 
continuity of~$\mu_k (B_0[f(t, \cdot)])$ and therefore~$\mu_k(f(t,\cdot))$ is continuously differentiable in time 
and satisfies~\eqref{eq-evolution-moments-f}.
\end{proof}

From now on, we suppose that Assumption~\ref{assumption-growth} is 
satisfied and that~$f^0$ (of positive mass~$\rho^0$) has a finite second moment. 
We are interested in the normalized population distribution. 
Let~$\rho(t)$ be the total population mass at time~$t$ 
\begin{equation*}
  \rho(t) = \mu_0(f) = \int_\mathbb{R} f(t,x) \, \d x.
\end{equation*}
From Proposition~\ref{prop-weak-formulation2-f}, we obtain the 
evolution equation for~$\rho(t)$ 
\begin{equation}
  \ddt \rho(t) = \rho(t) - \int_\mathbb{R} m(x) f(t,x)\, \d x.\label{dtrho}
\end{equation}
As noted in the proof of Proposition~\ref{prop-evolution-moments-f}, we have that~$\rho$ is positive for all time, 
and we can therefore consider the normalized population distribution $g$ that is defined in~\eqref{def-g}.

Thus, from Proposition~\ref{prop-weak-formulation2-f}, we obtain a weak formulation for the 
following equation on~$g$ 
\begin{equation*}
  \partial_t g(t,x) = B_0[g](t,x) - m(x)g(t,x) - \left(1- \int_\mathbb{R} m(z)g(t,z)\,\d z \right) g(t,x).
\end{equation*}

\begin{remark}
  This evolution equation for the normalized population distribution 
  has the same form as the replicator-mutator equation, with the linear mutator term 
  replaced by a sexual reproduction one. Replicator-mutator type equations have 
  been quite studied in the form of nonlocal integro-differential equations (see e.g.~\cite{Alf-Car-2014,Alf-Car-2017,Gil-Ham-al-2017}).
  Recently, the authors of~\cite{Clo-Gab-2022} have proved concentration results and new convergence estimates 
  in the non-singular case for measure solutions of the 
  “\emph{house of cards}” replicator-mutator model. 
\end{remark}

Then, we can also define the mean trait, or center of mass, which we denote by $\overline{x}$ and which is formulated 
by~\eqref{def-xbar}. Thus, we can focus on the study of the centered moments $M_k$ and the centered 
variations of the selection $S_k$, for~$k\in\mathbb{N}$, which formulations are given in~\eqref{def-Mk}-\eqref{def-Sk}.

Indeed, the moment~$M_k$ is well-defined for all time whenever~$f^0$ has a moment of order~$k$ and~$S_k$ is well-defined if~$f^0$ has a moment of order~$k+2$. From now on, we will write~$M_k^0=M_k(0)$.

In particular we have~$M_0(t)=1$ and~$M_1(t)=0$ for all time. 
Using~\eqref{eq-evolution-moments-f} (for~$k=1$) and~\eqref{dtrho} 
we obtain the equation of evolution of~$\overline{x}$ given by~\eqref{dtxbar}
whenever~$f^0$ has a moment of order~$3$ :
\begin{align*}
  \frac{\d \overline{x}}{\d t} = \ddt\big(\frac{\mu_1(f(t,\cdot))}{\rho(t)}\big)&=-\int_\mathbb{R}m(x)xg(t,x)\d x +\frac{\mu_1(f(t,\cdot))}{\rho(t)}\int_\mathbb{R}m(x)g(t,x)\d x \\
  &= \int m(x)(\overline{x}(t)-x)g(t,x)\d x = -S_1(t),
\end{align*}
which was claimed in~\eqref{dtxbar}.

From now on, we actually assume that~$f^0$ has a moment of order~$4$ 
(we will only compute estimates for even order moments), so that this 
mean trait~$\overline{x}$ is differentiable, and we can then compute 
the evolution of the centered moments~$M_k$, thanks 
to Proposition~\ref{prop-evolution-moments-f} (or 
Proposition~\ref{prop-continuity-moments-bounds}) and to 
the evolutions~\eqref{dtrho} and~\eqref{dtxbar} of~$\rho$ and~$\overline{x}$.

The next proposition states the differential equations and inequalities that~$M_k$ satisfies. They 
will be used in the following sections to prove stability results of Dirac masses. Notice that when~$f^0$ has a moment of order~$4$, this proposition provides the results claimed in~\eqref{dtM2}-\eqref{dtM4} for the evolutions of the moments~$M_2$ and~$M_4$.
\begin{proposition}\label{prop-dtMk}
  Let~$k\geqslant2$. Under Assumption~\ref{assumption-growth}, if~$f^0$ has a moment of order~$k+2$, then~$M_k$ is continuously differentiable in time and we have
  \begin{equation}\label{dtMk}
    \frac{\d}{\d t}M_k=-\big(1-\frac1{2^{k-1}}\big)M_k+\frac1{2^k}\sum_{i=2}^{k-2}\binom{k}{i}M_{k-i}M_i -S_k+S_0M_k+k S_1 M_{k-1}.
  \end{equation}
  
  If~$k\geqslant4$ is even and~$f^0$ has only a moment of order~$k$, then~$M_k$ is continuous in time and we have for all~$t\geqslant0$,
  \begin{equation}\label{subsolMk}
    \frac{\d^+}{\d t}M_k\leqslant-\Big(\frac12-\frac1{2^{k-1}}-S_0+\eta(\overline{x})\Big)M_k + \frac1{2^k}\sum_{i=2}^{k-2}\binom{k}{i}M_{k-i}M_i +k S_1 M_{k-1}.
  \end{equation}
  \end{proposition}

\begin{proof} Let us first remark that if we fix~$x_0\in\mathbb{R}$, then~$\widetilde{f}=f(\cdot,\cdot+x_0)$ satisfies the same equation with~$m$ replaced by~$\widetilde{m}=m(\cdot+x_0)$.
  \begin{align*}
    \rho(t)M_k(t)&=\int_\mathbb{R}(x-x_0+x_0-\overline{x}(t))^kf(t,x)\d x=\sum_{i=0}^k\binom{k}{i}(x_0-\overline{x}(t))^i\mu_{k-i}(\widetilde{f}(t,\cdot))\\
              &=\mu_k(\widetilde{f}(t,\cdot))-k(\overline{x}(t)-x_0)\mu_{k-1}(\widetilde f(t,\cdot))+\mathcal{O}(|\overline{x}(t)-x_0|^2),
  \end{align*}
where the bounds for the term~$\mathcal{O}(|\overline{x}(t)-x_0|^2)$ are uniform on~$[0,T]$, thanks to Proposition~\ref{moment-prop}.
Therefore if~$x_0=\overline{x}(t_0)$ for some~$t_0\in[0,T)$, since~$t\mapsto\mu_{k-1}(\widetilde{f}(t,\cdot))$ is continuous 
in time (and in particular at~$t_0$), and using~\eqref{dtxbar}, we have that
\[\frac{\rho(t)M_k(t)-\mu_k(\widetilde{f}(t,\cdot))}{t-t_0}\underset{t\to t_0}{\longrightarrow}-k\frac{\d\overline{x}}{\d t}(t_0)\mu_{k-1}(\widetilde{f}(t_0,\cdot))=kS_1(t_0)\rho(t_0)M_{k-1}(t_0).\]
Therefore, setting~$y(t)=\rho(t)M_k(t)-\mu_k(\widetilde{f}(t,\cdot))$, we get that~$y$ is differentiable at~$t_0$ with
\[\frac{\d y}{\d t}\big|_{t=t_0}=kS_1(t_0)\rho(t_0)M_{k-1}(t_0).\]

When~$f$ has a moment of order~$k+2$ (initially), this is the same for~$\widetilde{f}$ and therefore we can use Proposition~\ref{prop-evolution-moments-f} to get the derivative of~$\mu_k(\widetilde{f}(t,\cdot))$ at~$t_0$. We obtain
\begin{align*}\frac{\d}{\d t}\mu_k(\widetilde{f}(t,\cdot))\Big|_{t=t_0}&=\sum_{j=0}^k\binom{k}{j}\frac{\mu_j(\widetilde{f}(t_0,\cdot))\mu_{k-j}(\widetilde{f}(t_0,\cdot))}{2^k\mu_0(\widetilde{f}(t_0,\cdot))}-\int_\mathbb{R}\widetilde{m}(x)x^k\widetilde{f}(t_0,x)\d x\\
  &=\frac1{2^k}\sum_{j=0}^k\binom{k}{j}\rho(t_0)M_j(t_0)M_{k-j}(t_0)-\rho(t_0)\big(S_k(t_0)+m(\overline{x}(t_0))M_{k}(t_0)\big).
\end{align*}
Since~$M_k(t)=\frac1{\rho(t)}(y(t)+\mu_k(\widetilde{f}(t,\cdot)))$ and since~$\rho$ is differentiable at~$t_0$, we get that
\[\frac{\d}{\d t}M_k\Big|_{t=t_0}=\frac1{\rho(t_0)}\frac{\d}{\d t}(y(t)+\mu_k(\widetilde{f}(t,\cdot)))\Big|_{t=t_0}-\frac{M_k(t_0)}{\rho(t_0)}\frac{\d \rho}{\d t}\Big|_{t=t_0}.\]
The derivative of~$\rho$ at~$t_0$ is given in~\eqref{dtrho} by~$\rho(t_0)\big(1-S_0(t_0)-m(\overline{x}(t_0))\big)$, 
therefore we obtain~\eqref{dtMk} at time~$t_0$, using the fact that~$M_0(t_0)=1$ and~$M_1(t_0)=0$. 
Since the right-hand side of~\eqref{dtMk} is continuous in time, as proved in Proposition~\ref{prop-evolution-moments-f}, we obtain that~$M_k$ is continuously differentiable in time.

When~$f$ has only (initially) a moment of order~$k$ with~$k$ even and~$k\geqslant4$, we use similarly Proposition~\ref{prop-continuity-moments-bounds} instead of~\ref{prop-evolution-moments-f} to get, by setting~$K=-\inf_{\mathbb{R}}m$, 
\begin{align*}\frac{\d^+}{\d t}\mu_k(\widetilde{f}(t,\cdot))\Big|_{t=t_0}&\leqslant\sum_{j=0}^k\binom{k}{j}\frac{\mu_j(\widetilde{f}(t_0,\cdot))\mu_{k-j}(\widetilde{f}(t_0,\cdot))}{2^k\mu_0(\widetilde{f}(t_0,\cdot))}+K\mu_k(\widetilde{f}(t_0,\cdot))\\
  &\leqslant\frac1{2^k}\sum_{j=0}^k\binom{k}{j}\rho(t_0)M_j(t_0)M_{k-j}(t_0)+\rho(t_0)KM_k(t_0).
\end{align*}
Since~$k\geqslant4$,~$\rho$ and~$y$ are still differentiable at~$t_0$, and since~$\rho$ is positive, we get that
\[\frac{\d^+}{\d t}M_k\Big|_{t=t_0}=\frac1{\rho(t_0)}\Big(\frac{\d y}{\d t}\Big|_{t=t_0}+\frac{\d^+\mu_k(\widetilde{f}(t,\cdot))}{\d t}\Big|_{t=t_0}\Big)-\frac{M_k(t_0)}{\rho(t_0)}\frac{\d \rho}{\d t}\Big|_{t=t_0}.\]
The computations are therefore similar to the previous ones, and we obtain
\[ \frac{\d^+}{\d t}M_k\leqslant-\big(1-\frac1{2^{k-1}}-S_0-K-m(\overline{x})\big)M_k + \frac1{2^k}\sum_{i=2}^{k-2}\binom{k}{i}M_{k-i}M_i +k S_1 M_{k-1},\]
which corresponds to~\eqref{subsolMk}.
\end{proof}

\begin{remark}\label{remark-subsol-m-Lispschitz} If we suppose that~$m$ grows at most linearly, we only need to have a second moment to control the derivative of~$\rho$, and therefore~\eqref{subsolMk} is also valid for~$k=2$.
\end{remark}
\subsection{Fourier formulation, and rescaled profile}

Before tackling the long time behavior of~$\overline x$ and of the centered moments~$M_k$, we finish 
this section by stating the equation satisfied by the Fourier transform of the centered and rescaled formulation of~$g(t, \cdot)$. 
This equation will be further investigated in Section~\ref{section-4}.

\bigskip

From Definition~\ref{def-gamma}, we have 
\begin{equation}\label{momentsgamma}
\int_\mathbb{R}  \gamma(t,x) \, \d x = 1, 
\int_\mathbb{R} x \gamma(t,x) \, \d x = 0, \text{ and }
\int_\mathbb{R} x^2 \gamma(t,x) \, \d x = 1.
\end{equation}

We also define the signed measure~$\varphi$
\begin{equation}\label{eq-def-eps}
  \varphi(t,x) := [m \big( \sqrt{M_2(t)} x + \overline{x} (t) \big)-m(\overline{x}(t))]\gamma(t,x),
\end{equation}

Now we prove the following result.
\begin{proposition}
  Under Assumption~\ref{assumption-growth}, if~$f^0$ has a moment of order~$4$, then~$\widehat{\gamma}$ is differentiable and we have the equation on~$\widehat{\gamma}$ 
\begin{equation}\label{evolutionGammahat}
  \partial_t\widehat{\gamma}(t,\xi)=\widehat{\gamma}(t,\tfrac{\xi}2)^2-\widehat{\gamma}(t,\xi)+\frac14\xi\partial_\xi\widehat{\gamma}(t,\xi)+R(t,\xi),
\end{equation}
with~$R(t,\xi)$ given by
\begin{equation}\label{eq-defR}
  R(t,\xi):=- iS_1(t)\frac{\xi}{\sqrt{M_2(t)}} \widehat{\gamma}(t,\xi) - \widehat{\varphi}(t, \xi) 
  + S_0(t)\widehat{\gamma}(t,\xi) + \frac{1}{2}\left(\frac{S_2(t) }{M_2(t)} - S_0(t)\right) \xi \partial_\xi \widehat{\gamma}(t,\xi).
\end{equation}
\end{proposition}

\begin{proof}
From Assumption~\ref{assumption-growth} and Proposition~\ref{prop-weak-formulation2-f}, and assuming that~$f^0$ has 
finite moments up to order 4,~$\widehat{f}$ is differentiable in time, and using~\eqref{weak-form-B0}, we obtain 
\begin{equation*}
  \partial_t \widehat{f} (t,\xi)= \frac{\left( \widehat{f}\left(t, \xi / 2\right)\right)^2 }{\widehat{f}(t,0)}
  - \widehat{mf} (t,\xi).
\end{equation*}
Thus, we obtain the equation on~$\widehat{g}$, using~\eqref{dtrho}:
\begin{equation}\label{eq-fourier-g}
  \partial_t \widehat{g}(t, \xi) = \big( \widehat{g}( t,\xi / 2)\big)^2 - \widehat{g}(t,\xi)
  - \big(1-\widehat{mg}(t,0)\big)\widehat{g}(t,\xi).
\end{equation}
Then,~$\widehat{\gamma}$ is well defined and we have, for all~$\xi \in \R$,
\begin{equation*}
  \widehat{\gamma}(t,\xi) = e^{i\frac{\xi}{\sqrt{M_2(t)}}\overline{x} (t)}\widehat{g}\left(t,\frac{\xi}{\sqrt{M_2(t)}}\right).
\end{equation*}
Thus, from Proposition~\ref{prop-dtMk} and~\eqref{eq-fourier-g},~$\widehat{\gamma}$ is differentiable in time and we can compute 
\begin{align*}
  \partial_t \widehat{\gamma} (t,\xi) =  &i\frac{\xi}{\sqrt{M_2}}\frac{\d \overline{x}}{\d t} \widehat{\gamma}(t,\xi) - i\frac{\xi \overline{x}}{2M_2^{3/2}}\frac{\d M_2}{\d t} \widehat{\gamma}(t, \xi) + e^{i\frac{\xi}{\sqrt{M_2(t)}}\overline{x} (t)}\partial_t \widehat{g}(t,\xi/\sqrt{M_2})\\
  & - e^{i\frac{\xi}{\sqrt{M_2(t)}}\overline{x}}\frac{1}{2M_2^{3/2}}\frac{\d M_2}{\d t} \xi \partial_\xi \widehat{g}(t,\xi/\sqrt{M_2}),
\end{align*}
with~$\partial_\xi \widehat{g}$ well defined since~$g$ has moment of order one.
Using~\eqref{dtxbar} and~\eqref{eq-fourier-g}, we get 
\begin{align*}
  \partial_t \widehat{\gamma} (t,\xi) = &\left( \widehat{\gamma} (t,\xi/2)\right)^2 - \widehat{\gamma}(t,\xi) -e^{i\frac{\xi}{\sqrt{M_2(t)}}\overline{x} (t)}\widehat{mg}(t,\tfrac{\xi}{\sqrt{M_2}}) + \widehat{mg}(t,0) \widehat{\gamma}(t,\xi) \\
  &- iS_1(t)\frac{\xi}{\sqrt{M_2}} \widehat{\gamma}(t,\xi) - \frac{1}{2} \frac{\d M_2}{\d t}\frac{\xi }{M_2^{3/2}}\left(i\overline{x}\widehat{\gamma}(t,\xi) + e^{i\frac{\xi}{\sqrt{M_2(t)}}\overline{x} (t)}\partial_\xi \widehat{g}(t,\tfrac{\xi}{\sqrt{M_2}})\right).
\end{align*}
And using~\eqref{dtM2}, we then compute 
\begin{align*}
  \partial_t \widehat{\gamma} (t,\xi) = &\left( \widehat{\gamma} (t,\xi/2)\right)^2 - \widehat{\gamma}(t,\xi) -e^{i\frac{\xi}{\sqrt{M_2(t)}}\overline{x} (t)}\widehat{mg}(t,\tfrac{\xi}{\sqrt{M_2}}) + \widehat{mg}(t,0) \widehat{\gamma}(t,\xi) 
  - iS_1(t)\frac{\xi}{\sqrt{M_2}} \widehat{\gamma}(t,\xi) \\
  &+ \frac{\xi}{2}\left(\frac{1}{2}\frac{1 }{\sqrt{M_2}} + \frac{S_2 }{M_2^{3/2}} - \frac{S_0}{\sqrt{M_2}} \right) 
  \left(i\overline{x}\widehat{\gamma}(t,\xi) + e^{i\frac{\xi}{\sqrt{M_2(t)}}\overline{x} (t)}\partial_\xi \widehat{g}(t,\tfrac{\xi}{\sqrt{M_2}})\right).
\end{align*}
Then, noticing the expressions  
\begin{align*}
  &\partial_\xi \widehat{\gamma}(t,\xi) =\frac{1 }{\sqrt{M_2}} \left(i\overline{x}\widehat{\gamma}(t,\xi) + e^{i\frac{\xi}{\sqrt{M_2(t)}}\overline{x} (t)}\partial_\xi \widehat{g}(t,\tfrac{\xi}{\sqrt{M_2}})\right), \\
  &\widehat{mg}(t,0) = S_0(t) - m\left(\overline{x}(t)\right),
\end{align*}
the definition~\eqref{eq-def-eps} enables to write 
\begin{equation*}
  e^{i\frac{\xi}{\sqrt{M_2(t)}}\overline{x} (t)}\widehat{mg}(t,\tfrac{\xi}{\sqrt{M_2}}) - \widehat{mg}(t,0) \widehat{\gamma}(t,\xi) =
  \widehat{\varphi}(t, \xi) - S_0(t)\widehat{\gamma}(t,\xi),
\end{equation*}
and we obtain the equation on~$\widehat{\gamma}$
\begin{align*}
  \partial_t \widehat{\gamma} (t,\xi) = &\left( \widehat{\gamma} (t,\xi/2)\right)^2 - \widehat{\gamma}(t,\xi) + \frac{1}{4} \xi \partial_\xi \widehat{\gamma}(t,\xi)
  - iS_1(t)\frac{\xi}{\sqrt{M_2}} \widehat{\gamma}(t,\xi) - \widehat{\varphi}(t, \xi) \\
  &+ S_0(t)\widehat{\gamma}(t,\xi) + \frac{1}{2}\left(\frac{S_2 }{M_2} - S_0 \right) \xi \partial_\xi \widehat{\gamma}(t,\xi).
\end{align*}
Finally, with the definition of~$R$~\eqref{eq-defR}, we get~\eqref{evolutionGammahat}.
\end{proof}

\section{Long-time behaviour of normalized moments}\label{section-moments}
\subsection{Local stability of some Dirac masses}
We begin this section by proving the local stability result of Theorem~\ref{thm-stability-dirac}, which states that a sufficient condition for the 
convergence of~$g$ towards a Dirac mass is to be initially close enough
to the Dirac mass located at~$\overline{x}_0=\overline{x}(0)$, when~$\eta(\overline{x}_0)>0$, i.e.
when the initial selection rate at the center of mass satisfies~$m(\overline{x}_0)<\inf_{\mathbb{R}}m + \frac12$.

  Before that, we state an useful technical lemma

\begin{lemma}
  \label{lemma-assumption-m}
  We suppose that Assumption~\ref{assumption-local-lipschitz} is satisfied (that is to say $m$ is locally Lipschitz, bounded below and growing at most quadratically). Then, there exists a nonnegative function~$\alpha$, which is locally bounded on~$\mathbb{R}$, and a constant~$\beta\geqslant0$ such that for all~$x,y\in\mathbb{R}$, we have
  \begin{equation}\label{eq-assumption-m}
    |m(x)-m(y)|\leqslant\alpha(y)|x-y|+\beta|x-y|^2.
  \end{equation}
\end{lemma}
\begin{proof}
  For~$r>0$, let us denote~$L_r=\sup_{|x|,|y|\leqslant r, x\neq y}\frac{|m(x)-m(y)|}{|x-y|}$ the Lipschitz constant of~$m$ on~$[-r,r]$. We suppose that~$|m(x)|\leqslant \widetilde{C}(1+x^2)$ for all $x$ in~$\mathbb{R}$.

  Let us now pick $x$ and~$y$ in~$\mathbb{R}$. If $|x|\leqslant2|y|+1$, then
  \[|m(x)-m(y)|\leqslant L_{2|y|+1} |x-y|.\]
  Otherwise, we have $|x|\geqslant2|y|+1$ and therefore~$|x-y|\geqslant|x|-|y|\geqslant1+|y|\geqslant\max(1,|y|)$. Using~$|y|\leqslant\frac{|x|-1}{2}$ we also get~$|x-y|\geqslant|x|-|y|\geqslant|x|-\frac{|x|-1}{2}\geqslant \frac{|x|}{2},$
  so finally we obtain
  \[|m(x)-m(y)|\leqslant |m(x)|+|m(y)|\leqslant\widetilde{C}(2+x^2+y^2)\leqslant\widetilde{C}(2|x-y|^2+4|x-y|^2+|x-y|^2)\leqslant 7\widetilde{C}|x-y|^2.\]
  In both cases, we obtain~\eqref{eq-assumption-m} with~$\alpha(y)=L_{2|y|+1}$ and~$\beta=7\widetilde{C}$.
  
  Notice that conversely, if $m$ (bounded below) satisfies the estimate~\eqref{eq-assumption-m}, then Assumption~\ref{assumption-local-lipschitz} is true.
\end{proof}

We are now ready to prove Theorem~\ref{thm-stability-dirac}.
\begin{proof}[Proof of Theorem~\ref{thm-stability-dirac}]
  Thanks to the definitions~\eqref{def-Mk},~\eqref{def-Sk} and~\eqref{def-eta}, we obtain
  \begin{equation}
    S_2\geqslant\big(\eta(\overline{x})-\frac12\big)M_2,\quad S_4\geqslant\big(\eta(\overline{x})-\frac12\big)M_4.\label{eq-S2M2S4M4}
  \end{equation}
  By Lemma~\ref{lemma-assumption-m}, using Cauchy-Schwarz inequalities in the definitions~\eqref{def-Mk}-\eqref{def-Sk}, we obtain

  \begin{align}
    M_2&\leqslant\sqrt{M_4},\label{eq-M2M4}\\
    |M_3|&\leqslant\sqrt{M_2M_4},\label{eq-M3sqM2M4}\\
    |S_0|&\leqslant\alpha(\overline{x})\sqrt{M_2}+\beta M_2,\label{eq-S0M2}\\
    |S_1|&\leqslant\alpha(\overline{x})M_2+\beta\sqrt{M_2M_4}\label{eq-S1M2M4}.
  \end{align}
  
  Therefore we obtain, thanks to~\eqref{dtxbar},~\eqref{dtM2}, and~\eqref{dtM4}:
  \begin{align}
    \label{estim_xbar} \Big|\frac{\d\overline{x}}{\d t}\Big|&\leqslant\alpha(\overline{x})M_2+\beta\sqrt{M_2M_4},\\
    \label{estim_dtM2} \frac{\d}{\d t} M_2&\leqslant(-\eta(\overline{x})+\alpha(\overline{x})\sqrt{M_2}+\beta M_2)M_2,\\
    \label{estim_dtM4} \frac{\d^+}{\d t} M_4&\leqslant\big(-\eta(\overline{x})+5(\alpha(\overline{x})\sqrt{M_2}+\beta M_2)\big)M_4,
  \end{align}
  where we have used~\eqref{eq-M2M4} and~\eqref{eq-M3sqM2M4} to get that~$|M_2M_3|\leqslant\sqrt{M_4}\sqrt{M_2M_4}=\sqrt{M_2}M_4$.

  We now fix~$\delta_0$ such that~$0<\delta<\delta_0<\eta(\overline{x}_0)$, and by 
  continuity of~$m$ we find~$r>0$ such that~$\eta(x)\geqslant\delta_0$ for 
  all~$x\in[\overline{x}_0-r,\overline{x}_0+r]$. 
  We set~$\widetilde{\alpha}=\sup_{[\overline{x}_0-r,\overline{x}_0+r]}\alpha$, 
  which is finite by local boundedness of~$\alpha$. Using~\eqref{eq-M2M4}, we then fix~$M_4^0$ sufficiently small in order to have
  \begin{gather*}
    5\big(\widetilde{\alpha}\sqrt{M_2^0}+\beta M_2^0\big)<\delta_0-\delta,\\
    \widetilde{\alpha}M_2^0+\beta\sqrt{M_2^0M_4^0}<r\delta.
  \end{gather*}
  Finally we define~$T$ as the supremum of all~$t>0$ such that for all~$s\in[0,t]$, we have
  \begin{equation}\label{conditions-T}
    \begin{cases}
       |\overline{x}(s)-\overline{x}_0|\leqslant r,\\
     5\widetilde{\alpha}\sqrt{M_2(s)}+\beta M_2(s)\leqslant\delta_0-\delta,\\
      \widetilde{\alpha}M_2(s)+\beta\sqrt{M_2(s)M_4(s)}\leqslant r\delta,
    \end{cases}
  \end{equation}
  and the set of such~$t>0$ is non empty, therefore~$T>0$ (we may have~$T=+\infty$).

  For all~$t\in[0,T)$, we get from~\eqref{estim_dtM2}-\eqref{estim_dtM4} that~$\frac{\d}{\d t} M_2\leqslant-\delta M_2$ and~$\frac{\d^+}{\d t} M_4\leqslant-\delta M_4$, giving that~$M_2$ and~$M_4$ are nonincreasing in time and satisfying estimates~\eqref{M2-decay-delta2}-\eqref{M4-decay-delta2}, thanks to Lemma~\ref{lemma-dplus}. Thanks to~\eqref{estim_xbar}, we obtain
  \begin{align*}
    |\overline{x}(t)-\overline{x}_0|&\leqslant\int_0^t(\widetilde{\alpha}M_2^0+\beta\sqrt{M_2^0M_4^0})e^{-\delta s}\d s\\
                                    &\leqslant(\widetilde{\alpha}M_2^0+\beta\sqrt{M_2^0M_4^0})\frac{1-e^{-\delta t}}{\delta}<(1-e^{-\delta t})r.
  \end{align*}
Let us prove that~$T=+\infty$ by contradiction. If~$T$ is finite, then~$|\overline{x}(T)-\overline{x}_0|\leqslant(1-e^{-\delta T})r<r$, and by monotonicity of~$M_2$ and~$M_4$ we also get
  \begin{align*}
    & 5\widetilde{\alpha}\sqrt{M_2(T)}+\beta M_2(T)\leqslant5\widetilde{\alpha}\sqrt{M_2^0}+\beta M_2^0<\delta_0-\delta,\\
    & \widetilde{\alpha}M_2(T)+\beta\sqrt{M_2(T)M_4(T)}\leqslant\widetilde{\alpha}M_2^0+\beta\sqrt{M_2^0M_4^0}<r\delta.
  \end{align*}

  Therefore, by continuity in time of~$M_2$,~$M_4$ and~$\overline{x}$, there exists~$\tau>0$ such that for all~$s$ in~$[T-\tau,T+\tau]$, we have the estimations~\eqref{conditions-T}. Thanks to the definition of~$T$, these estimations are true for all~$s\in[0,T)$, and therefore also for all~$s\in[0,T+\tau]$, in contradiction with the definition of~$T$ as a supremum. 

  Since the derivative of~$\overline{x}$ is decaying exponentially fast, we get that~$\overline{x}(t)$ converges to some~$\overline{x}_\infty$ in~$\mathbb{R}$, and more precisely we have
  \begin{equation*}
    |\overline{x}(t)-\overline{x}_\infty|\leqslant\int_t^\infty(\widetilde{\alpha}M_2^0+\beta\sqrt{M_2^0M_4^0})e^{-\delta s}\d s=\frac{\widetilde{\alpha}M_2^0+\beta\sqrt{M_2^0M_4^0}}{\delta}e^{-\delta t}.
  \end{equation*}
  This also shows that the smaller the value of~$M_4^0$, the closer the location of~$\overline{x}_\infty$ relatively to~$\overline{x}_0$.

  Finally, the convergence in Wasserstein distance comes from the fact that~$W_4(g, \delta_{\overline{x}})=M_4(t)^{\frac14}$, and therefore
  \[W_4(g, \delta_{\overline{x}_\infty})\leqslant(M_4^0)^{\frac14}e^{-\frac{\delta}{4}t}+W_4(\delta_{\overline{x}(t)}, \delta_{\overline{x}_\infty})=(M_4^0)^{\frac14}e^{-\frac{\delta}{4}t}+|\overline{x}(t)-\overline{x}_\infty|.\qedhere \]
\end{proof}

\begin{remark} If~$m$ is Lipschitz (we can take $\beta=0$ and~$\alpha$ constant in Lemma~\ref{lemma-assumption-m}), then we only need to suppose~$M_2^0$ sufficiently small to get the estimates~\eqref{eta-xbar-delta0},~\eqref{M2-decay-delta2} and the exponential convergence of~$\overline{x}$. The proof is exactly the same, only using the estimate on the evolution of~$M_2$ given by~\eqref{subsolMk} (see Remark~\ref{remark-subsol-m-Lispschitz}), and the convergence is in~$2$-Wasserstein distance.
\end{remark}

\subsection{Improved estimates on even order moments}
We can now use this result of stability of Theorem~\ref{thm-stability-dirac} together with the equation for higher order moments to improve the bound on the rate of convergence. In this subsection, we start from the assumptions made in Theorem~\ref{thm-stability-dirac}:

\begin{assumption}\label{assumption-improvements} We suppose that 
  the selection function~$m$ satisfies Assumption~\ref{assumption-local-lipschitz}, 
  and we take~$\overline{x}_0$ and $\delta$ such that~$0<\delta<\eta(\overline{x}_0)$. 
  Finally, we suppose that the initial profile~$g_0$ is such 
  that~$\overline{x}(0)=\overline{x}_0$ and that~$M_4^0$ is sufficiently 
  small, ensuring that the conclusion of Theorem~\ref{thm-stability-dirac} is true. As in the proof of this theorem, we denote by~$\widetilde{\alpha}$ a uniform bound in time for~$\alpha(\overline{x}(t))$ (independent of~$M_4^0$, only depending on~$\eta$,~$\delta$ and~$\overline{x}_0$).
\end{assumption}

To obtain the estimates on~$M_{2k}$, the idea is to use the equations~\eqref{dtMk} and~\eqref{subsolMk}, along with the stability result
stated in Theorem~\ref{thm-stability-dirac}. 
Apart from the exponential decay rates of the moments, we will also need a control on the constants in front of these estimates. 
Furthermore we will often have to deal with quantities controlled by~$e^{-\widetilde{\lambda}t}$ for all~$\widetilde{\lambda}$ strictly less than some~$\lambda>0$ to be identified.
Therefore we introduce the following definition.
\begin{definition}\label{def-decay-controlled}
  For~$\lambda>0$, we say that a nonnegative quantity~$y(t)$ (which may depend on several parameters) has a decay controlled by the parameter~$\sigma\geqslant0$ and with rate~$[\lambda]^-$ if for all~$\widetilde{\lambda}<\lambda$, there exists a constant~$C(\sigma)>0$ converging to~$0$ when~$\sigma\to0$, and such that for all~$t>0$, we have~$y(t)\leqslant C(\sigma)e^{-\widetilde{\lambda}t}$.
\end{definition}

We will use several times the following lemma.
\begin{lemma}\label{lemma-ineq-diff}
  If~$y$ is a nonnegative continuous function in time satisfying for all~$t>0$ 
  \begin{equation*}
    \frac{\d^+}{\d t}y(t)\leqslant(-\lambda+v_1(t))y(t)+v_0(t),
  \end{equation*}
  where~$v_ 0(t)$ and~$v_1(t)$ have a decay controlled by the parameter~$\sigma$ and with respective rates~$[\omega_0]^-$ and~$[\omega_1]^-$, then~$y(t)$ has a decay controlled by~$\max(y(0),\sigma)$ and with rate~$[\min(\lambda,\omega_0)]^-$.
\end{lemma}

\begin{proof}
  If~$\widetilde{\lambda}<\min(\lambda,\omega_0)$, we fix~$\widetilde{\omega}_0$ and~$\widetilde{\omega}_1$ such that~$\widetilde{\lambda}<\widetilde{\omega}_0<\omega_0$ and~$0<\widetilde{\omega}_1<\omega_1$, and we have for all~$t>0$
  \[\frac{\d^+}{\d t}y(t)\leqslant(-\widetilde{\lambda}+C_1(\sigma)e^{-\widetilde{\omega}_1t})y(t)+C_2(\sigma)e^{-\widetilde{\omega}_0t},\]
  with~$C_1(\sigma)$ and~$C_2(\sigma)$ converging to~$0$ as~$\sigma\to0$.
  Therefore, thanks to Lemma~\ref{lemma-dplus}, we obtain
  
  \begin{align*}
    y(t)e^{\widetilde{\lambda}t+\frac{C_1(\sigma)}{\widetilde{\omega}_1}(e^{-\widetilde{\omega}_1t}-1)}&\leqslant y(0)+C_2(\sigma)\int_0^te^{\widetilde{\lambda}s+\frac{C_1(\sigma)}{\widetilde{\omega}_1}(e^{-\widetilde{\omega}_1s}-1)}e^{-\widetilde{\omega}_0s}\d s\\
                                                                              &\leqslant y(0)+C_2(\sigma)\int_0^te^{\widetilde{\lambda}s-\widetilde{\omega}_0s}\d s=y(0)+\frac{C_2(\sigma)}{\widetilde{\omega}_0-\widetilde{\lambda}}(1-e^{-(\widetilde{\omega}_0-\widetilde{\lambda})t})\\
                                                                              &\leqslant y(0)+\frac{C_2(\sigma)}{\widetilde{\omega}_0-\widetilde{\lambda}}.
  \end{align*}
Finally we get
  \[ y(t)\leqslant\Big(y(0)+\frac{C_2(\sigma)}{\widetilde{\omega}_0-\widetilde{\lambda}}\Big)e^{\frac{C_1(\sigma)}{\widetilde{\omega}_1}(1-e^{-\widetilde{\omega}_1t})-\widetilde{\lambda}t}\leqslant\Big(y(0)+\frac{C_2(\sigma)}{\widetilde{\omega}_0-\widetilde{\lambda}}\Big)e^{\frac{C_1(\sigma)}{\widetilde{\omega}_1}}\cdot e^{-\widetilde{\lambda}t},\]
  and this ends the proof, since the constant converges to~$0$ as~$y(0)$ and~$\sigma$ converge to~$0$.
\end{proof}

First, we obtain the next proposition that presents links between rates of convergence of~$M_2$ and~$M_4$, allowing to improve the rates provided by Theorem~\ref{thm-stability-dirac}.
\begin{proposition}\label{prop-M2M4-firstimprovement}
  Under Assumption~\ref{assumption-improvements} we have that~$M_2$ and~$M_4$ have decay controlled by~$M_4^0$ and with respective rates~$[\min(\frac12,\frac38+\delta)]^-$ and~$[\frac38+\delta]^-$.
\end{proposition}
\begin{proof} We first suppose that~$M_2$ and~$M_4$ have decay controlled by~$M_4^0$ and with respective rates~$[\lambda_2]^-$ and~$[\lambda_4]^-$, for~$\lambda_2,\lambda_4>0$, and show that these rates can be improved. Thanks to Lemma~\ref{lemma-assumption-m}, using once again Cauchy-Schwarz inequality in the definition~\eqref{def-Sk}, we obtain
  \begin{equation}
    |S_2|\leqslant\widetilde{\alpha}\sqrt{M_2M_4}+\beta M_4\label{eq-S2M2M4}.
  \end{equation}
  Therefore we obtain, thanks to~\eqref{dtM2} and the estimate~\eqref{eq-S0M2},
  \begin{align*}
    \frac{\d}{\d t} M_2&\leqslant\left(-\frac12+\widetilde{\alpha}\sqrt{M_2}+\beta M_2\right)M_2+\widetilde{\alpha}\sqrt{M_2M_4}+\beta M_4,\\
                       &\leqslant\left(-\frac12+\varepsilon+\widetilde{\alpha}\sqrt{M_2}+\beta M_2\right)M_2+\left(\frac{\widetilde{\alpha}^2}{4\varepsilon}+\beta\right)M_4,
  \end{align*}
  where we used Young’s inequality for some~$\varepsilon>0$. Similarly, using~\eqref{dtM4},~\eqref{eq-S2M2S4M4} and~\eqref{eta-xbar-delta0} we obtain
  \begin{equation*}
    \frac{\d^+}{\d t} M_4\leqslant\left(-\delta-\frac38+5(\widetilde{\alpha}\sqrt{M_2}+\beta M_2)\right)M_4+\frac38M_2^2.
  \end{equation*}
  Therefore, since~$\widetilde{\alpha}\sqrt{M_2}+\beta M_2$,~$M_4$ and~$M_2^2$ have decay controlled by the initial moment~$M_4^0$ and respective rates~$[\frac12\lambda_2]^-$,~$[\lambda_4]^-$ and~$[2\lambda_2]^-$, by applying Lemma~\ref{lemma-ineq-diff},~$M_4$ have decay controlled by~$M_4^0$ with rate~$[\min(\frac38+\delta,2\lambda_2)]^-$, while~$M_2$ has decay controlled by~$\max(M_2^0,M_4^0)$ (and therefore, thanks to~\eqref{eq-M2M4}, by~$M_4^0$ only) and with rate~$[\min(\frac12-\varepsilon,\lambda_4)]^-$. Since this is true for any~$\varepsilon>0$, we obtain that~$M_2$ has rate~$[\min(\frac12,\lambda_4)]^-$. Therefore, by applying once more this property, we obtain that~$M_2$ and~$M_4$ have decay controlled by~$M_4^0$ and with respective rates~$[\min(\frac12,\frac38+\delta,2\lambda_2)]^-$ and~$[\min(\frac38+\delta,2\lambda_4)]^-$. Indeed~$\min(\frac38+\delta,1,2\lambda_4)=\min(\frac38+\delta,2\lambda_4)$ since~$\delta<\eta(\overline{x}_0)\leqslant\frac12$.
  
  Hence, since~$M_2$ and~$M_4$ both have decay controlled by~$M_4^0$ and with rate~$[\delta]^-$, thanks to Theorem~\ref{thm-stability-dirac} and the estimate~\eqref{eq-M2M4}, we directly obtain by induction that for all~$k\in\mathbb{N}$,~$M_2$ and~$M_4$ have decay controlled by~$M_4^0$ and with respective rates~$[\min(\frac12,\frac38+\delta,2^k\delta)]^-$ and~$[\min(\frac38+\delta,2^{k}\delta)]^-$, so for~$k$ sufficiently large, this ends the proof.
\end{proof}

The next proposition uses these estimates on~$M_2$ and~$M_4$ to obtain estimates for the rates of convergence of all moments of even order, provided they are finite initially.

\begin{proposition}\label{prop-estimates-Mk}
   Under Assumption~\ref{assumption-improvements}, for all~$k\geqslant2$, if~$M_{2k}^0$ is finite, then~$M_{2k}(t)$ has decay controlled by~$M_{2k}^0$ and with rate~$[\frac12-\frac1{2^{2k-1}}+\delta]^-$.
\end{proposition}
\begin{proof}
  We proceed by induction, since, the case~$k=2$ is provided by Proposition~\ref{prop-M2M4-firstimprovement}.  We fix~$k>2$, and suppose the result is true for~$k-1$.

  First of all, if~$2i\leqslant\ell\leqslant2j$, using Hölder inequality in the definition~\eqref{def-Mk}, we get
  \begin{align}
    |M_\ell|&\leqslant\int_{\mathbb{R}}|x-\overline{x}(t)|^\ell g(x)\d x=\int_{\mathbb{R}}|x-\overline{x}(t)|^{2i\frac{2j-\ell}{2j-2i}+2j\frac{\ell-2i}{2j-2i}}g(x)\d x\nonumber\\
         &\leqslant M_{2i}^{\frac{2j-\ell}{2j-2i}}M_{2j}^{\frac{\ell-2i}{2j-2i}}.\label{holder-Mi}
  \end{align}
  In particular, for~$i=0$ and~$j=k$, we obtain
  \begin{equation}
    M_2\leqslant M_{2k}^{\frac1k},\quad M_4\leqslant M_{2k}^{\frac2k},\quad M_{2k-2}\leqslant M_{2k}^{1-\frac1k},\quad \text{and}\quad |M_{2k-1}|\leqslant M_{2k}^{1-\frac1{2k}}.\label{eqs-M2k}
  \end{equation}
Using~\eqref{eq-S1M2M4} to get
  \begin{equation}
    |S_1M_{2k-1}|\leqslant(\widetilde{\alpha}M_2+\beta\sqrt{M_2M_4})M_{2k}^{1-\frac1{2k}}\leqslant(\widetilde{\alpha}\sqrt{M_2}+\beta\sqrt{M_4})M_{2k}.\label{S1M2km1}
  \end{equation}
  As previously, thanks to the definitions~\eqref{def-Mk},~\eqref{def-Sk},~\eqref{def-eta} and the property~\eqref{eta-xbar-delta0}, we obtain
  \begin{equation}
    S_{2k}\geqslant\big(\delta-\frac12\big)M_{2k}.\label{S2k-delta0}
  \end{equation}
  Now, using~\eqref{holder-Mi} with~$i=1$ and~$j=k-1$, we obtain, when~$2\leqslant\ell\leqslant2k-2$:
  \begin{equation}\label{MlM2kml}|M_\ell||M_{2k-\ell}|\leqslant M_2M_{2k-2},
  \end{equation}
  and therefore, we use~\eqref{S2k-delta0},~\eqref{eq-S0M2}, and~\eqref{S1M2km1} in the differential inequality~\eqref{subsolMk} for~$M_{2k}$ to get
  \begin{equation}\label{estimate-dtM2k}
    \frac{\d^+}{\d t}M_{2k}\leqslant\big(-\tfrac12+\tfrac{1}{2^{2k-1}}-\delta+v_1(t)\big)M_{2k}+\Big(\tfrac1{2^{2k}}\sum_{\ell=2}^{2k-2}\tbinom{2k}{\ell}\Big)M_2M_{2k-2},
  \end{equation}
  where~$v_1(t)=(2k+1)\widetilde{\alpha}\sqrt{M_2}+\beta(2k\sqrt{M_4}+M_2)$. Now, using~\eqref{eqs-M2k}, we obtain that when~$M_{2k}^0$ converges to~$0$, then~$M_4^0$ and~$M_2^0$ also converge to~$0$. Therefore by induction, and thanks to Proposition~\ref{prop-M2M4-firstimprovement},~$v_1(t)$ and~$M_2M_{2k-2}$ have decay controlled by~$M_{2k}^0$ and with respective rates~$[\min(\frac14,\frac3{16}+\frac12\delta)]^-$ and~$[\omega_0]^-$, where
  \begin{equation}\label{eq-rate-M2M2km2}
    \omega_0=\big(\min(\tfrac12,\tfrac38+\delta)+(\tfrac12-\tfrac1{2^{2k-3}}+\delta)\big)>(\tfrac38+\tfrac12-\tfrac1{2^{2k-3}}+\delta)\geqslant \tfrac34+\delta,
  \end{equation}
  since~$k\geqslant3$. Therefore, since~$\frac34+\delta > \frac12-\frac1{2^{2k-1}}+\delta$, we can apply Lemma~\ref{lemma-ineq-diff} and this ends the proof.
\end{proof}

We can now use the estimate on a moment of high order to obtain better rates for all lower order moments, as stated in the following proposition, which corresponds, with Definition~\ref{def-decay-controlled}, to the claimed Proposition~\ref{prop-improved-estimates-intro} in the introduction.

\begin{proposition}\label{prop-improvements-Mk}
  We suppose that Assumption~\ref{assumption-improvements} is satisfied. If~$M_{2k_0}^0$ is finite (for~$k_0\geqslant2$), then for all~$k\leqslant k_0$,~$M_{2k}$ has decay controlled by~$M_{2k_0}^0$ with rate~$[\min(1-\frac1{2^{2k-1}},\frac12-\frac1{2^{2k_0-1}}+\delta)]^-$.
\end{proposition}
\begin{proof}
  We proceed once more by induction, from~$k=k_0$ down to~$k=1$. Indeed, since we have~$\delta<\eta(\overline{x}_0)\leqslant\frac12$, we get that~$\min(1-\frac1{2^{2k_0-1}},\frac12-\frac1{2^{2k_0-1}}+\delta)=\frac12-\frac1{2^{2k_0-1}}+\delta$, which corresponds to the rate given by Proposition~\ref{prop-estimates-Mk}. We fix~$k<k_0$ and suppose that the result is true for~$k+1$.
  We have, thanks to Lemma~\ref{lemma-assumption-m} and using Young’s inequality for some~$\varepsilon>0$:
  \begin{align*}
    |S_{2k}|&\leqslant\widetilde{\alpha}\int_\mathbb{R}|x-\overline{x}(t)|^{2k+1}g(x)\d x+\beta M_{2k+2}\\
            &\leqslant\int_\mathbb{R}|x-\overline{x}(t)|^{2k}\Big(\varepsilon+\frac{\widetilde{\alpha}^2}{4\varepsilon}|x-\overline{x}(t)|^2\Big)g(x)\d x+\beta M_{2k+2}  \leqslant\varepsilon M_{2k}+\left(\frac{\widetilde{\alpha}^2}{4\varepsilon}+\beta\right)M_{2k+2}.
  \end{align*}
  Therefore, as in the proof of Proposition~\ref{prop-estimates-Mk}, using~\eqref{MlM2kml},~\eqref{eq-S0M2}, and~\eqref{S1M2km1} in the differential equation~\eqref{dtMk} for~$M_{2k}$, we get
  \begin{equation*}
    \frac{\d}{\d t}M_{2k}\leqslant\big(-1+\tfrac{1}{2^{2k-1}}+\varepsilon+v_1(t)\big)M_{2k}+\left(\tfrac1{2^{2k}}\sum_{\ell=2}^{2k-2}\tbinom{2k}{\ell}\right)M_2M_{2k-2}+\left(\frac{\widetilde{\alpha}^2}{4\varepsilon}+\beta\right)M_{2k+2},
  \end{equation*}
  where~$v_1(t)=(2k+1)\widetilde{\alpha}\sqrt{M_2}+\beta(2k\sqrt{M_4}+M_2)$ has decay controlled by~$M_4^0$ (and hence by~$M_{2k_0}^0$) and rate~$[\min(\frac14,\frac3{16}+\frac12\delta)]^-$.  By induction~$M_{2k+2}$ has decay controlled by~$M_{2k_0}^0$ and with rate~$[\widetilde{\omega}_0]^-$ where~$\widetilde{\omega}_0=\min(1-\frac1{2^{2k+1}},\frac12-\frac1{2^{2k_0-1}}+\delta)$. And we know by Propositions~\ref{prop-M2M4-firstimprovement} and~\ref{prop-estimates-Mk} that~$M_2M_{2k-2}$ has decay controlled by~$M_{2k}^0$ (and hence by~$M_{2k_0}^0$) and with rate~$[\omega_0]^-$ with~$\omega_0$ given in~\eqref{eq-rate-M2M2km2}, so when~$k\geqslant3$ we have~$\omega_0>\frac34+\delta>\frac12-\frac1{2^{2k_0-1}}+\delta\geqslant\widetilde{\omega}_0$.

  When~$k=2$ we also obtain~$\omega_0=\min(\frac34+2\delta,1)>\frac12-\frac1{2^{2k_0-1}}+\delta\geqslant\widetilde{\omega}_0$ since~$\delta<\frac12$. Finally when~$k=1$ the sum on~$\ell$ is empty.

  So, in all cases we get thanks to Lemma~\ref{lemma-ineq-diff} that~$M_{2k}$ has decay controlled by~$\max(M_{2k}^0,M_{2k_0}^0)$ (and therefore by~$M_{2k_0}^0$ only) and with rate~$[\min(1-\tfrac{1}{2^{2k-1}}-\varepsilon,\widetilde{\omega}_0)]^-$, for any~$\varepsilon>0$, which ends the proof.
\end{proof}

When the initial condition has enough moments, we then see thanks to Proposition~\ref{prop-improvements-Mk} that all~$M_{2k}$ with~$k\geqslant2$ decay with rates strictly greater than~$\frac12$, and that~$M_2$ decays with rate~$[\frac12]^-$. Therefore we expect all the moments~$M_{2k}$ to decay “faster” than~$M_2$, and we will see in the next section that this property is important to study the self-similar behaviour of the solution as time goes to infinity. However, to have such a property, we have to get a lower bound for~$M_2(t)$. By using a refinement of Lemma~\ref{lemma-ineq-diff}, we could prove that there exists~$C_2\geqslant0$ such that~$|M_2(t)-C_2e^{-\frac12t}|\leqslant Ce^{-\omega_0t}$ with~$\omega_0>\frac12$, but this is not sufficient to get a lower bound on~$M_2$, as we could have~$C_2=0$.

To overcome this difficulty, instead of asking the fourth moment~$M_4$ to be small 
initially, we will ask~$\frac{M_{2k_0}}{M_2}$ to be small initially, for~$k_0$ 
sufficiently large. This is for instance the case if we shrink the initial profile 
around~$\overline{x}_0$ by a given parameter~$\frac1{\varepsilon}$, which has the effect of scaling 
any initial moment~$M_{2k}$ by a factor~$\varepsilon^{2k}$. Before proving that this condition is 
enough to get lower bounds on~$M_2$, we will need an adaptation of Lemma~\ref{lemma-ineq-diff} 
in a nonlinear setting (the main difference being that we now require smallness of the initial condition~$y(0)$ and of the parameter~$\sigma$),
which in stated in the following result.
\begin{lemma}\label{lemma-ineq-diff-nonlin}
  We suppose that~$y$ is a nonnegative continuous function satisfying for all~$t>0$ :
  \begin{equation*}
    \frac{\d^+}{\d t}y(t)\leqslant\big(-\lambda+v_1(y(t))\big)y(t)+v_0(t),
  \end{equation*}
  where~$v_1(r)\leqslant Cr^\theta$ for all~$r\geqslant0$ sufficiently small (with~$\theta>0$ and~$C>0$), and~$v_0(t)$ has a decay controlled by the parameter~$\sigma$ and with rate~$[\omega_0]^-$.

  Then, if~$y(0)$ and~$\sigma$ are small enough,~$y(t)$ has a decay controlled by~$\max(y(0),\sigma)$ and with rate~$[\min(\lambda,\omega_0)]^-$.
\end{lemma}

\begin{proof}
  We fix~$\widetilde{\omega}_0<\omega_0$, then~$\widetilde{\lambda}<\min(\lambda,\widetilde{\omega}_0)$, and we pick~$r>0$ such that~$\lambda-Cr^\theta>\widetilde{\lambda}$ (and small enough so that~$v_1(\widetilde{r})\leqslant C\widetilde{r}^\theta$ as soon as~$\widetilde{r}\leqslant r$). Therefore, as soon as~$y(t)\leqslant r$, we have
  \[\frac{\d^+}{\d t}y(t)\leqslant-\widetilde{\lambda}y(t)+\widetilde{C}(\sigma)e^{-\widetilde{\omega}_0t},\]
  where~$\widetilde{C}(\sigma)\to0$ as~$\sigma\to0$. We suppose that~$y(0)+\frac{\widetilde{C}(\sigma)}{\widetilde{\omega}_0-\widetilde{\lambda}}\leqslant r$.

  We denote~$T=\sup\{t>0,\forall s\in[0,t],y(s)\leqslant r\}$. By continuity of~$y$, since~$y(0)<r$, we have~$T>0$ and therefore for all~$t\in[0,T]$,
  \[y(t)\leqslant\Big(y(0)+\widetilde{C}(\sigma)\frac{1-e^{-(\widetilde{\omega}_0-\widetilde{\lambda})t}}{\widetilde{\omega}_0-\widetilde{\lambda}}\Big)e^{-\widetilde{\lambda}t}\leqslant re^{-\widetilde{\lambda}t},\]
  showing that~$T=+\infty$. Indeed, if it was not the case, we would have~$y(T)<r$, in contradiction with the definition of~$T$, by continuity of~$y$. Therefore we also obtain
  \[v_1(y(t))\leqslant C\Big(y(0)+\widetilde{C}(\sigma)\frac{1}{\widetilde{\omega}_0-\widetilde{\lambda}}\Big)^\theta e^{-\theta\widetilde{\lambda}t},\]
  showing that~$v_1(y(t))$ has a rate controlled by~$\max(y(0),\sigma)$ with rate~$[\theta\widetilde{\lambda}]^-$. We can therefore apply Lemma~\ref{lemma-ineq-diff-nonlin}, as we now have a linear estimate (provided~$y(0)+\frac{\widetilde{C}(\sigma)}{\widetilde{\omega}_0-\widetilde{\lambda}}\leqslant r$), and this ends the proof.
\end{proof}

We are now ready to prove the following result, which leads to the statements of Theorem~\ref{thm-moments-decay}.
\begin{proposition}\label{prop-moments-decay}
  Under Assumption~\ref{assumption-improvements}, if~$\frac{M_{2k_0}^0}{M_2^0}$ is sufficiently small and if~$\frac1{2^{2k_0-1}}<\delta$, then~$\frac{M_{2k_0}}{M_2}$ has decay controlled by its initial value~$\frac{M_{2k_0}^0}{M_2^0}$ and with rate~$[\delta-\frac1{2^{2k_0-1}}]^-$.

  Furthermore, in that case, there exists~$C_{k_0}(\frac{M_{2k_0}^0}{M_2^0})$, converging to~$1$ as~$\frac{M_{2k_0}^0}{M_2^0}\to0$ such that
  \[M_2(t)\geqslant C_{k_0}\Big(\frac{M_{2k_0}^0}{M_2^0}\Big)M_2^0\,e^{-\frac{t}2}.\]

  Consequently, for all~$k\geqslant2$ (with~$k\leqslant k_0$),~$\frac{M_{2k}}{M_2}$ has decay controlled by~$\frac{M_{2k_0}^0}{M_2^0}$ and with rate~$[\min(\frac12-\frac1{2^{2k-1}},\delta-\frac1{2^{2k_0-1}})]^-$.
\end{proposition}

\begin{proof}
We compute
  \[    \frac{\d^+}{\d t} \big(\frac{M_{2k_0}}{M_2}\big)=\frac1{M_2}\frac{\d^+ M_{2k_0}}{\d t}  - \frac1{M_2}\frac{\d M_{2}}{\d t} \big(\frac{M_{2k_0}}{M_2}\big).\]
  Thanks to~\eqref{dtM2} and the estimates~\eqref{eq-S0M2}~\eqref{eq-S2M2M4}, we have
  \begin{equation}\label{dtM2-lowerbound}
    \frac{\d}{\d t} M_2\geqslant(-\frac12-\widetilde{\alpha}\sqrt{M_2}-\beta M_2)M_2-\widetilde{\alpha}\sqrt{M_2M_4}-\beta M_4,
  \end{equation}
  and therefore, using~\eqref{estimate-dtM2k}, since~$\tfrac1{2^{2k}}\sum_{\ell=2}^{2k-2}\tbinom{2k}{\ell}\leqslant1$, we obtain
  \begin{equation}\label{dtM2k0}
    \frac{\d^+}{\d t}\big(\frac{M_{2k_0}}{M_2}\big)\leqslant\big(-\delta+\tfrac{1}{2^{2k_0-1}}+\widetilde{w}(t)\big)\frac{M_{2k_0}}{M_2}+M_{2{k_0}-2},
  \end{equation}
  where
  \begin{equation*}
    \widetilde{w}(t)=\widetilde{\alpha}\big((2k+2)\sqrt{M_2}+\sqrt{\frac{M_4}{M_2}}\big)+\beta\big(2k\sqrt{M_4}+2M_2+\frac{M_4}{M_2}\big).
  \end{equation*}
  Using~\eqref{eq-M2M4}, we get
  \begin{equation}\label{eq-M2-M4surM2}
    M_2\leqslant\sqrt{M_4}=\frac{M_4}{\sqrt{M_4}}\leqslant\frac{M_4}{M_2},
  \end{equation}
  and therefore we obtain
  \begin{equation*}
 \widetilde{w}(t)\leqslant(2k+3)\widetilde{\alpha}\,\sqrt{\frac{M_4}{M_2}}+(2k+3)\beta\,\frac{M_4}{M_2}.
  \end{equation*}
    
  Now, using~\eqref{holder-Mi} with~$i=1$,~$\ell=4$ and~$j=k_0$, we obtain
  \begin{equation}\label{M4surM2}
    \frac{M_4}{M_2}\leqslant\frac{M_2^{\frac{2k_0-4}{2k_0-2}}M_{2k_0}^{\frac{2}{2k_0-2}}}{M_2}=\big(\frac{M_{2k_0}}{M_2}\big)^{\frac{2}{2k_0-2}},
  \end{equation}
  and therefore we get
  \begin{equation}\label{estimate-wtilde}
    \widetilde{w}(t) \leqslant(2k+3)\widetilde{\alpha}\big(\frac{M_{2k_0}}{M_2}\big)^{\frac{1}{2k_0-2}} +(2k+3)\beta\big(\frac{M_{2k_0}}{M_2}\big)^{\frac{2}{2k_0-2}}.
  \end{equation}

  Finally, thanks to Proposition~\ref{prop-improvements-Mk},~$M_{2k_0-2}$ has decay controlled by the initial moment~$M_{2k_0}^0$ and with rate~$[\min(1-\frac1{2^{2k_0-3}},\frac12-\frac1{2^{2k_0-1}}+\delta)]^-$, and since we have
  \[M_{2k_0}\leqslant\big(\frac{M_{2k_0}}{M_2}\big)\cdot M_2\leqslant\big(\frac{M_{2k_0}}{M_2}\big)^{1+\frac{2}{2k_0-2}},\]
  we get that~$M_{2k_0}^0\to0$ if~$\frac{M_{2k_0}^0}{M_2^0}\to0$, so~$M_{2k_0-2}$ has decay controlled by~$\frac{M_{2k_0}^0}{M_2^0}$.
  Thanks to~\eqref{dtM2k0} and~\eqref{estimate-wtilde}, we can apply Lemma~\ref{lemma-ineq-diff-nonlin} with~$y(t)=\frac{M_{2k_0}}{M_2}$ and~$\sigma=\frac{M_{2k_0}^0}{M_2^0}=y(0)$, the function~$w$ being given by~$v_1(r)=(2k+3)(\widetilde{\alpha}\,r^{\theta}+\beta\,r^{2\theta})$, where~$\theta=\frac{1}{2k_0-2}$. We therefore obtain that if~$\frac{M_{2k_0}^0}{M_2^0}$ is sufficiently small, then~$\frac{M_{2k_0}}{M_2}$ has decay controlled by its initial value~$\frac{M_{2k_0}^0}{M_2^0}$ and with rate~$[\delta-\tfrac{1}{2^{2k_0-1}}]^-$.

  Still denoting~$\sigma=\frac{M_{2k_0}^0}{M_2^0}$, if we pick~$\omega_0<\delta-\tfrac{1}{2^{2k_0-1}}$, we then get that there exists~$C(\sigma)$, converging to~$0$ as~$\sigma\to0$, such that we have
  \[\frac{M_{2k_0}}{M_2}\leqslant C(\sigma) e^{-\omega_0t}.\]
  Therefore, using~\eqref{dtM2-lowerbound},~\eqref{eq-M2-M4surM2} and~\eqref{M4surM2}, we obtain
  \begin{align*}
    \frac{\d}{\d t}M_2&\geqslant\Big(-\frac12 -2\widetilde{\alpha}\sqrt{\frac{M_4}{M_2}}-2\beta\frac{M_4}{M_2}\Big)\,M_2\\
                      &\geqslant\Big(-\frac12 -2\widetilde{\alpha}\big(\frac{M_{2k_0}}{M_2}\big)^{\theta}-2\beta\big(\frac{M_{2k_0}}{M_2}\big)^{2\theta}\Big)\,M_2\\
    &\geqslant\Big(-\frac12 - \widetilde{C}(\sigma)\,e^{-\theta\omega_0t} \Big)\,M_2,
  \end{align*}
  where
  \[\widetilde{C}(\sigma)=2\widetilde{\alpha}C(\sigma)^{\theta}+2\beta C(\sigma)^{2\theta},\]
  which still converges to~$0$ as~$\sigma\to0$. Solving this differential inequality for~$M_2$, we get
  \[M_2(t)\geqslant e^{-\frac{t}2+\frac{\widetilde{C}(\sigma)}{\omega_0\theta}(e^{-\omega_0\theta t}-1)}M_2^0\geqslant e^{-\frac{\widetilde{C}(\sigma)}{\omega_0\theta}}M_2^0 e^{-\frac{t}2},\]
  which gives the expected lower bound on~$M_2$.

  The last part of the proposition is a direct application of Proposition~\ref{prop-improvements-Mk}, using the fact that~$M_{2k_0}\to0$ as soon as~$\frac{M_{2k_0}^0}{M_2^0}\to0$.
\end{proof}

\section{Convergence to a self-similar profile}\label{section-4}

The goal of this section is to prove Theorem~\ref{th-cv-gamma}, 
which states the large time convergence of the centered and rescaled 
measure~$\gamma$ to a 
self-similar profile in the Fourier distance defined in~\eqref{def-fourier-distance}, for some value of~$s$.

We first discuss some prior results obtained on the following equation 
\begin{equation}\label{eq_gamma0}
  \partial_t\widehat{\gamma}(t,\xi)=\widehat{\gamma}^2\left(t,\tfrac{\xi}2 \right)-\widehat{\gamma}(t,\xi)+\frac14\xi\partial_\xi\widehat{\gamma}(t,\xi),
\end{equation}
which is precisely equation~\eqref{evolutionGammahat} in the case~$R \equiv 0$, that is when~$m$ is 
constant and then there is no selection. This equation 
on the Fourier transform of the rescaled distribution can also be derived from the kinetic model 
of Boltzmann-type studied in~\cite{Par-Tos-2006}. In this work, the authors proved a contraction property 
of the operator of equation~\eqref{eq_gamma0} in the space of probability measures that satisfy~\eqref{momentsgamma}
and have finite moments up to order or~$s = 2+ \delta$ with~$\delta>0$ well chosen,
endowed with the Fourier distance~$d_s$. Also, they proved 
the convergence in the Fourier distance of the rescaled measure towards the unique stationary solution of~\eqref{eq_gamma0} 
 with mass~$1$, centered, and second moment
equal to~$1$, which is denoted by~$\gamma_\infty$ and defined in~\eqref{def-gamma-inf}.

Here is the outline of the proof. First, we derive an estimate
on the distance between~$\gamma$ and~$\gamma_\infty$, under the assumption that 
the term~$R(t,\xi)$ exponentially decreases in time and is also controlled by a~$|\xi|^s$.
Next, we prove that, for all~$(t,\xi) \in \R_+ \times \R$,~$R(t,\xi)$ satisfy theses 
conditions in the framework we set. We eventually combine these results to conclude to the 
statement of Theorem~\ref{th-cv-gamma}.

The first step of the proof is the following result.
\begin{proposition} \label{prop-convergence-gamma}
  We fix~$s\in(2,3)$ and we set~$\lambda_s=1-\frac{s}4-2^{1-s}$ 
  (we have~$\lambda_s>0$ since~$s\mapsto\lambda_s$ is strictly concave 
  and~$\lambda_2=\lambda_3=0$). We suppose there exist~$L>0$ and~$c>0$ such that
  \begin{equation}\label{assum_R_evolution}
    |R (t,\xi)|\leqslant|\xi|^sLe^{-ct}, \quad \forall \xi \in \R.
  \end{equation}

  Then, for all time~$t$,  
\begin{itemize}
  \item if~$c=\lambda_s$, then we have
   \begin{equation*}
    d_s(\gamma,\gamma_\infty)(t)\leqslant d_s(\gamma_0,\gamma_\infty)e^{-\lambda_st}+Lt e^{-\lambda_s t} ,
  \end{equation*}
  \item otherwise, we have 
  \begin{equation*}
    d_s(\gamma,\gamma_\infty)(t)\leqslant d_s(\gamma_0,\gamma_\infty)e^{-\lambda_st}+L \frac{e^{-c t} - e^{-\lambda_s t}}{\lambda_s - c}.
  \end{equation*}
\end{itemize}
\end{proposition}

\begin{proof}
  From~\eqref{evolutionGammahat}, we obtain
  \[
    \frac{\d}{\d t}\left(e^t\widehat{\gamma}(t,e^{-\frac{t}4}\xi)\right)=e^t\widehat{\gamma}^2\left(t,e^{-\frac{t}4}\tfrac{\xi}2\right)+e^tR\left(t,e^{-\frac{t}4}\xi\right), 
  \]
  and so we get the following Duhamel formulation
  \begin{equation}\label{DuhamelR}
    e^t\widehat{\gamma}\left(t,e^{-\frac{t}4}\xi\right)=\widehat{\gamma_0}+\int_0^te^\tau\widehat{\gamma}^2\left(\tau,e^{-\frac{\tau}4}\tfrac{\xi}2\right)\d \tau+\int_0^te^\tau R\left(\tau,e^{-\frac{\tau}4}\xi\right)\d \tau. 
  \end{equation}

  Since~$\gamma_\infty$ is a stationary solution to~\eqref{evolutionGammahat} with~$R=0$, applying then in this case~\eqref{DuhamelR} with~$\widehat{\gamma}(t,\xi)=\widehat{\gamma_\infty}(\xi)$, we obtain
  \begin{equation}\label{DuhamelStationary}
    e^t\widehat{\gamma_\infty}\left(e^{-\frac{t}4}\xi\right)=\widehat{\gamma_\infty}+\int_0^te^\tau\widehat{\gamma_\infty}^2 \left(e^{-\frac{\tau}4}\tfrac{\xi}2\right)\d \tau.
  \end{equation}

  First, notice that
  \begin{align*} 
    |\widehat{\gamma}^2\left(\tau,e^{-\frac{\tau}4}\tfrac{\xi}2\right)-\widehat{\gamma_\infty}^2\left(e^{-\frac{\tau}4}\tfrac{\xi}2\right)|
    &\leqslant|\widehat{\gamma}\left(\tau,e^{-\frac{\tau}4}\tfrac{\xi}2 \right)+\widehat{\gamma_\infty}\left(e^{-\frac{\tau}4}\tfrac{\xi}2 \right)||\widehat{\gamma}\left(\tau,e^{-\frac{\tau}4}\tfrac{\xi}2\right)-\widehat{\gamma_\infty}\left(e^{-\frac{\tau}4}\tfrac{\xi}2\right)|\\
    &\leqslant2|e^{-\frac{\tau}4}\tfrac{\xi}2|^sd_s(\gamma,\gamma_\infty)(\tau)=2^{1-s}e^{-s\frac{\tau}4}|\xi|^sd_s(\gamma,\gamma_\infty)(\tau)\,
  \end{align*}
  since~$\|\widehat{\gamma}\|_\infty\leqslant1$ and~$\|\widehat{\gamma_\infty}\|_\infty\leqslant1$ ($\gamma$ and~$\gamma_\infty$ are probability measures). 
  
\medskip

  Besides, we have from~\eqref{assum_R_evolution}
  \[|R(\tau,e^{-\frac{\tau}4}\xi)|\leqslant e^{-\frac{s\tau}4}|\xi|^sLe^{-c\tau}. \]
Therefore, we obtain the estimate on the last term in formula~\eqref{DuhamelR}
  \[\int_0^te^\tau R\left(\tau,e^{-\frac{\tau}4}\xi\right)\d \tau\leqslant L|\xi|^s\int_0^te^{(1-\frac{s}4-c)\tau}\d \tau .\]

  Now, by subtracting~\eqref{DuhamelStationary} from~\eqref{DuhamelR} and dividing by~$|\xi|^s$, we obtain
  \begin{equation*} 
    e^t\frac{\left|\widehat{\gamma}\left(t,e^{-\frac{t}4}\xi\right)-\widehat{\gamma_\infty}\left(e^{-\frac{t}4}\xi \right)\right|}{|\xi|^s}\leqslant d_s(\gamma_0,\gamma_\infty)+2^{1-s}\int_0^te^{(1-\frac{s}4)\tau}d_s(\gamma,\gamma_\infty)(\tau)\d \tau+ L\int_0^te^{(1-\frac{s}4-c)\tau}\d \tau.
  \end{equation*}
  Since we assumed that~$f^0$ has finite moment up to order 4 and 
  also~$\gamma_\infty$ up to order~$s$, the distance~$d_s(\gamma (t,\cdot), \gamma_\infty)$ is finite, then we can take the supremum on~$\xi\neq0$ in the estimate above to finally obtain
  \begin{equation}\label{gronwall-ds}
    e^{(1-\frac{s}4)t}d_s(\gamma,\gamma_\infty)(t)\leqslant d_s(\gamma_0,\gamma_\infty)+2^{1-s}\int_0^te^{(1-\frac{s}4)\tau}d_s(\gamma,\gamma_\infty)(\tau)\d \tau+  L\int_0^te^{(1-\frac{s}4-c)\tau}\d \tau.
  \end{equation}
  This Grönwall estimate can be solved classically by setting
  \[z(t)=e^{-2^{1-s}t}\Big(d_s(\gamma_0,\gamma_\infty)+2^{1-s}\int_0^te^{(1-\frac{s}4)\tau}d_s(\gamma,\gamma_\infty)(\tau)\d \tau\Big),\]
  thus we obtain
  \[z'(t)=2^{1-s}e^{-2^{1-s}t}e^{(1-\frac{s}4)t}d_s(\gamma,\gamma_\infty)(t)-2^{1-s}z(t)\leqslant2^{1-s}Le^{-2^{1-s}t}\int_0^te^{(1-\frac{s}4-c)\tau}\d \tau,\]
  which gives, since~$z(0)=d_s(\gamma_0,\gamma_\infty)$,
\begin{equation}\label{estimate_z}
  z(t)\leqslant d_s(\gamma_0,\gamma_\infty)+L\left(-e^{-2^{1-s}t}\int_0^te^{(1-\frac{s}4-c)\tau}\d \tau + \int_0^te^{(\lambda_s-c)\tau}\d \tau\right).
\end{equation}
  
  Getting back to~\eqref{gronwall-ds}, we obtain
  \begin{equation*}
    d_s(\gamma,\gamma_\infty)(t) \leqslant e^{-(1-\frac{s}4)t}\Big(e^{2^{1-s}t}z(t) + L\int_0^te^{(1-\frac{s}4-c)\tau}\d \tau\Big)
  \end{equation*}
and using~\eqref{estimate_z}, we get
  \begin{multline*}
    d_s(\gamma,\gamma_\infty)(t)\leqslant d_s(\gamma_0,\gamma_\infty)e^{-\lambda_s t} + L\left( e^{-\lambda_s t}\int_0^te^{(\lambda_s-c)\tau}\d \tau - e^{-(1-\frac{s}4)t}\int_0^te^{(1-\frac{s}4-c)\tau}\d \tau\right) \\
    +L e^{-(1-\frac{s}4)t}\int_0^te^{(1-\frac{s}4-c)\tau} \d \tau.
  \end{multline*}
Hence we have 
\begin{equation*}
  d_s(\gamma,\gamma_\infty)(t)\leqslant d_s(\gamma_0,\gamma_\infty)e^{-\lambda_s t} + Le^{-\lambda_s t}\int_0^te^{(\lambda_s-c)\tau}\d \tau,
\end{equation*}
which gives the result stated in Proposition~\ref{prop-convergence-gamma}.
\end{proof}

Next, we establish some estimates on~$R(t,\xi)$ and the moments of~$g$ that will help to derive an estimate of 
the type~\eqref{assum_R_evolution}. From now on, we use the notation 
\begin{equation*}
  \widehat{\phi}'(t,\xi):= \frac{\partial \widehat{\phi}}{{\partial \xi}}(t,\xi), \quad 
  \widehat{\phi}''(t,\xi):= \frac{\partial^2 \widehat{\phi}}{{\partial \xi^2}}(t,\xi), \quad 
  \widehat{\phi}'''(t,\xi):= \frac{\partial^3 \widehat{\phi}}{{\partial \xi^3}}(t,\xi),
\end{equation*}
for the sake of simplicity. We will also omit the dependency on time in the proofs 
of the two next stated results, since all the estimates are derived at a fixed time~$t$. We will resume 
the dependency on time when we prove Theorem~\ref{th-cv-gamma}.

The lemma below provides some useful estimates on~$R$.
\begin{lemma}\label{lemma-estimates-Rgammahat}
  We fix~$t\geqslant0$. For all~$\xi\in\mathbb{R}$, we have the two following estimates:
  \begin{align}
    \label{estimate-Rgammahat-xi2}
    |R(t,\xi)|&\leqslant(\tfrac12\|\widehat{\varphi}(t,\cdot)\|_\infty+\|\widehat{\varphi}'(t,\cdot)\|_\infty+\|\widehat{\varphi}''(t,\cdot)\|_\infty)|\xi|^2,\\
    \label{estimate-Rgammahat-xi3}   |R(t,\xi)|&\leqslant\big(\tfrac1{12}(\|\widehat{\varphi}(t,\cdot)\|_\infty+3\|\widehat{\varphi}''(t,\cdot)\|_\infty)\|\widehat{\gamma}'''(t,\cdot)\|_\infty+\tfrac12\|\widehat{\varphi}'(t,\cdot)\|_\infty+\tfrac16\|\widehat{\varphi}'''(t,\cdot)\|_\infty\big)|\xi|^3.                      
  \end{align}
\end{lemma}
\begin{proof} From the definition~\eqref{eq-defR} of~$R$, we obtain 
  \begin{equation}\label{Rgammahat} 
    R(t,\xi)=S_0(t)\widehat{\gamma}(t,\xi)-\widehat{\varphi}(t,\xi)-\frac12\left(S_0(t)-\frac{S_2(t)}{M_2(t)}\right)\xi\widehat{\gamma}'(t,\xi)-i\frac{S_1(t)}{\sqrt{M_2(t)}}\xi\widehat{\gamma}(t,\xi).
  \end{equation}

  Since~$\gamma(t,\cdot)$ is a centered probability density with second moment~$1$, we get
  \begin{equation*}
    \widehat{\gamma}(t,0)=1,\quad  \widehat{\gamma}'(t,0)=0, \quad  \widehat{\gamma}''(t,0)=-1.
  \end{equation*}
  Furthermore, for all~$\xi$, we have
  \begin{equation}\label{est-gammahat-second}
    |\widehat{\gamma}(t,\xi)|\leqslant\int_\mathbb{R}\gamma(x)\d x=1,\quad |\widehat{\gamma}''(t,\xi)|=|-\widehat{x^2\gamma}(t,\xi)|\leqslant\int_\mathbb{R}|x|^2\gamma(t,x)\d x=1.
  \end{equation}
  We also have, by Cauchy-Schwarz inequality,
  \begin{equation}\label{est-gammahat-prime}
    |\widehat{\gamma}'(t,\xi)|=|-i\widehat{x\gamma}(t,\xi)|\leqslant\int_\mathbb{R}|x|\gamma(t,x)\d x\leqslant1.
  \end{equation}

  Now, a simple change of variable reads, for~$k \in \mathbb{N}$,
  \begin{equation*}
    \frac{S_k}{\sqrt{M_2}^k}=\int_\mathbb{R}x^k\varphi(x)\d x=i^k\widehat{\varphi}^{(k)}(0),
  \end{equation*}
  (where the dependency on time~$t$ is omitted) and therefore, the expression~\eqref{Rgammahat} of~$R(t,\xi)$ can be written under the form
  \begin{equation*}
    R(t,\xi)=\widehat{\varphi}(0)\widehat{\gamma}(\xi)-\widehat{\varphi}(\xi)+\Big(\widehat{\varphi}'(0)\widehat{\gamma}(\xi)-\tfrac12\big(\widehat{\varphi}(0)+\widehat{\varphi}''(0)\big)\widehat{\gamma}'(\xi)\Big)\xi.
  \end{equation*}
  From this, we easily remark that~$R(0)=0$ and we obtain
  \begin{equation}\label{eq-Rgammahat1} \begin{split}
    R'(t,\xi)=\widehat{\varphi}'(0)\widehat{\gamma}(\xi)-\widehat{\varphi}'(\xi)&+\tfrac12\big(\widehat{\varphi}(0)-\widehat{\varphi}''(0)\big)\widehat{\gamma}'(\xi) \\
      &+ \Big(\widehat{\varphi}'(0)\widehat{\gamma}'(\xi)-\tfrac12\big(\widehat{\varphi}(0)+\widehat{\varphi}''(0)\big)\widehat{\gamma}''(\xi)\Big)\xi.
    \end{split}
  \end{equation}
Also, we remark once again that~$R'(t,0)=0$ (since~$\widehat{\gamma}'(0)=0$), and we also have from~\eqref{eq-Rgammahat1}
\begin{equation*}
  \begin{split}
    R'(t,\xi)=\int_0^{\xi}\Big(\widehat{\varphi}'(0)\widehat{\gamma}'(\eta) &-\widehat{\varphi}''(\eta)+\tfrac12\big(\widehat{\varphi}(0)-\widehat{\varphi}''(0)\big) \widehat{\gamma}''(\eta)\Big) \d \eta \\
      &+ \Big(\widehat{\varphi}'(0)\widehat{\gamma}'(\xi)-\tfrac12\big(\widehat{\varphi}(0)+\widehat{\varphi}''(0)\big)\widehat{\gamma}''(\xi)\Big)\xi.
    \end{split}
  \end{equation*}

  Using~\eqref{est-gammahat-second}-\eqref{est-gammahat-prime}, we obtain a first estimate on~$R'$ 
  \[|R'(t,\xi)|\leqslant(\|\widehat{\varphi}\|_\infty+2\|\widehat{\varphi}'\|_\infty+2\|\widehat{\varphi}''\|_\infty)|\xi|,\]
  and since~$R(t,0)=0$, by integration we get the estimation~\eqref{estimate-Rgammahat-xi2}.

\bigskip

  Similarly, we compute the second derivative of~$R$ and we obtain
  \begin{equation*}
    \begin{split}
    R''(t,\xi)=2\widehat{\varphi}'(0)\widehat{\gamma}'(\xi)&-\widehat{\varphi}''(\xi)-\widehat{\varphi}''(0)\widehat{\gamma}''(\xi)\\
      &+ \Big(\widehat{\varphi}'(0)\widehat{\gamma}''(\xi)-\tfrac12\big(\widehat{\varphi}(0)+\widehat{\varphi}''(0)\big)\widehat{\gamma}'''(\xi)\Big)\xi.
    \end{split}
  \end{equation*}
  From this, we remark once more that~$R''(t,0)=0$ (since~$\widehat{\gamma}''(0)=-1$), and we also have
  \begin{equation*}
    \begin{split}
    R''(t,\xi)=\int_0^{\xi}\big(2\widehat{\varphi}'(0)\widehat{\gamma}''(\eta)&-\widehat{\varphi}'''(\eta)-\widehat{\varphi}''(0)\widehat{\gamma}'''(\eta)\big) \d \eta \\
      &+ \Big(\widehat{\varphi}'(0)\widehat{\gamma}''(\xi)-\tfrac12\big(\widehat{\varphi}(0)+\widehat{\varphi}''(0)\big)\widehat{\gamma}'''(\xi)\Big)\xi.
    \end{split}
  \end{equation*}
  Using again~\eqref{est-gammahat-second}-\eqref{est-gammahat-prime}, we obtain
  \[|R''(t,\xi)|\leqslant\big(\tfrac12(\|\widehat{\varphi}\|_\infty+3\|\widehat{\varphi}''\|_\infty)\|\widehat{\gamma}'''\|_\infty+3\|\widehat{\varphi}'\|_\infty+\|\widehat{\varphi}'''\|_\infty)|\xi|,\]
  and since~$R'(t,0)=R(t,0)=0$, by two successive integrations we get the estimation~\eqref{estimate-Rgammahat-xi3} and this ends the proof.
\end{proof}
Next, we prove estimates on the derivatives of~$\varphi$ with the moments~$M_k$, using the results of 
Section~\ref{section-moments}. Once more, we fix~$t\geqslant0$ and omit the dependency on~$t$ in the notations, the norm~$\|\cdot\|_\infty$ being the supremum with respect to the Fourier variable~$\xi$ only.

\begin{proposition}\label{prop-estimates-eps-lip} We have
  \begin{equation}\label{est-gammahat-third}
    \|\widehat{\gamma}'''\|_\infty\leqslant\frac{\sqrt{M_4}}{M_2}, 
  \end{equation}
and if~$m$ satisfies Assumption~\ref{assumption-local-lipschitz}, we have (with~$\alpha$ and~$\beta$ given by Lemma~\ref{lemma-assumption-m}) :
  \begin{align*}
    \|\widehat{\varphi}^{(k)}\|_\infty&\leqslant \alpha(\overline{x})\sqrt{\frac{M_4}{M_2}}+\beta\frac{M_4}{M_2},&\text{for }0\leqslant k\leqslant2, \\
    \|\widehat{\varphi}^{(k)}\|_\infty&\leqslant\alpha(\overline{x})\frac{M_4}{M_2^{\frac32}}+\beta\frac{\sqrt{M_4M_6}}{M_2^{\frac32}},&\text{for }0\leqslant k\leqslant3.
  \end{align*}
\end{proposition}
\begin{proof}
  By a simple change of variable, we have
  \begin{equation*}
    \frac{M_k}{\sqrt{M_2}^k}=\int_\mathbb{R}x^k\gamma(x)\d x,
  \end{equation*}
  and in particular, we get~\eqref{est-gammahat-third} by
  \begin{equation*}
    |\widehat{\gamma}'''(\xi)|\leqslant\int_\mathbb{R}|x|^3\gamma(x)\d x\leqslant\sqrt{\int_\mathbb{R}x^2\gamma(x)\d x}\sqrt{\int_\mathbb{R}x^4\gamma(x)\d x}=\sqrt{1}\sqrt{\frac{M_4}{\sqrt{M_2}^4}}=\frac{\sqrt{M_4}}{M_2}.
  \end{equation*}

  Furthermore, the quantity~$\int_\mathbb{R}|x|^k\gamma(x)\d x$ is nondecreasing with~$k$, from~$k\geqslant1$. Indeed, by induction, thanks to~\eqref{est-gammahat-prime} we first have~$\int_\mathbb{R}|x|\gamma(x)\d x\leqslant1=\int_{\mathbb{R}}|x|^2\gamma(x)\d x$. And if for some~$k\geqslant1$, we have~$\int_\mathbb{R}|x|^{k}\gamma(x)\d x\leqslant\int_{\mathbb{R}}|x|^{k+1}\gamma(x)\d x$
  then by Cauchy-Schwarz inequality, we get:
  \begin{align*}
    \int_\mathbb{R}|x|^{k+1}\gamma(x)\d x&\leqslant\sqrt{\int_\mathbb{R}|x|^{k}\gamma(x)\d x}\sqrt{\int_\mathbb{R}|x|^{k+2}\gamma(x)\d x}\\
    &\leqslant\sqrt{\int_\mathbb{R}|x|^{k+1}\gamma(x)\d x}\sqrt{\int_\mathbb{R}|x|^{k+2}\gamma(x)\d x},
  \end{align*}
  which proves that~$\int_\mathbb{R}|x|^{k+1}\gamma(x)\d x\leqslant\int_{\mathbb{R}}|x|^{k+2}\gamma(x)\d x$.

  Now, with Lemma~\ref{lemma-assumption-m}, we obtain by the definition~\eqref{eq-def-eps} that for all~$x\in\mathbb{R}$, we have
  \begin{equation*}
    |\varphi|\leqslant[\alpha(\overline{x})\sqrt{M_2}|x|+\beta M_2|x|^2]\gamma(x).
  \end{equation*}
  Therefore we obtain
  \begin{align*}
    |\widehat{\varphi}^{(k)}(\xi)|&\leqslant\int_\mathbb{R}|x|^k|\varphi(x)|\d x\nonumber\\
                           &\leqslant\alpha(\overline{x})\sqrt{M_2}\int_{\mathbb{R}}|x|^{k+1}\gamma(x)\d x +\beta M_2\int_{\mathbb{R}}|x|^{k+2}\gamma(x)\d x,
  \end{align*}
  which the gives
  \begin{align*}
    \|\widehat{\varphi}^{(k)}\|_\infty&\leqslant \alpha(\overline{x})\sqrt{M_2}\int_{\mathbb{R}}|x|^{3}\gamma(x)\d x+\beta M_2\int_{\mathbb{R}}|x|^{4}\gamma(x)\d x&\text{for }0\leqslant k\leqslant2,\\
    \|\widehat{\varphi}^{(k)}\|_\infty&\leqslant\alpha(\overline{x})\sqrt{M_2}\int_{\mathbb{R}}|x|^{4}\gamma(x)\d x+\beta M_2\int_{\mathbb{R}}|x|^{5}\gamma(x)\d x&\text{for }0\leqslant k\leqslant3.
  \end{align*}

  We have~~$\int_{\mathbb{R}}|x|^3\gamma(x)\d x\leqslant\frac{\sqrt{M_4}}{M_2}$ thanks to~\eqref{est-gammahat-third},~$\int_{\mathbb{R}}|x|^4\gamma(x)\d x=\frac{M_4}{M_2^2}$, and finally, by Cauchy-Schwarz inequality,
  \[\int_{\mathbb{R}}|x|^5\gamma(x)\d x\leqslant\sqrt{\int_{\mathbb{R}}|x|^4\gamma(x)\d x\int_{\mathbb{R}}|x|^6\gamma(x)}\d x\leqslant\sqrt{\frac{M_4M_6}{M_2^2M_2^3}},\]
  and we therefore obtain the desired estimates.
\end{proof}

We can conclude with the proof of Theorem~\ref{th-cv-gamma}.

\begin{proof}[Proof of Theorem~\ref{th-cv-gamma}]
Combining the 
results of Lemma~\ref{lemma-estimates-Rgammahat} and 
Proposition~\ref{prop-estimates-eps-lip} gives the following 
estimates for~$|R(t,\xi)|$:
\begin{equation*}
  |R(t,\xi)|\leqslant\frac52\left(\alpha(\overline{x})\sqrt{\frac{M_4}{M_2}}+\beta\frac{M_4}{M_2}\right)|\xi|^2,
\end{equation*}
and 
\begin{align*}
  |R(t,\xi)|&\leqslant\left( \alpha(\overline{x})\frac{M_4}{M_2^{\frac{3}2}}+\beta\left(\frac13\frac{M_4^{\frac32}}{M_2^2}+\frac23\frac{\sqrt{M_4M_6}}{M_2^{\frac32}}\right)\right)|\xi|^3\nonumber\\
  &\leqslant\left( \alpha(\overline{x})\frac{M_4}{M_2^{\frac{3}2}}+\beta\frac{\sqrt{M_4M_6}}{M_2^{\frac32}}\right)|\xi|^3,
\end{align*}
with the last inequality coming from the fact that~$M_4\leqslant\sqrt{M_2M_6}$ by Cauchy-Schwarz inequality.

\medskip

Then, we fix~$\delta$ from Assumption~\ref{assumption-improvements}. By Proposition~\ref{prop-moments-decay}, assuming that~$\frac{M_{2k_0}^0}{M_2^0}$
is small enough for some~$k_0 \geqslant 3$ (and such that~$\frac{1}{2^{2k_0 - 1}} < \delta$), we have that, for each~$\omega < \min(\frac38, \delta - \frac{1}{2^{2k_0 -1}})$,
there exist constants~$C_1$ and~$C_2$ such that
\begin{equation*}
  |R(t,\xi)|\leqslant C_1e^{-\frac{\omega}{2}t}|\xi|^2,
\end{equation*}
and
\begin{equation*}
  |R(t,\xi)|\leqslant C_2e^{\left(\frac{1}{4}-\omega\right)t}|\xi|^3.
\end{equation*}

Thus, taking~$s\in (2,3)$ and defining~$c_s:=(3-s)\frac{\omega}{2} -(s-2)(\frac{1}{4}-\omega)$, we obtain 
\begin{equation*}
  |R(t,\xi)| = |R(t,\xi)|^{(3-s)+(s-2)} \leqslant C_1^{3-s} C_2^{s-2} e^{-c_st}|\xi|^s,
\end{equation*}
with~$c_s>0$ when~$s<\bar{s}:=2+\frac{\omega}{\frac{1}{2}-\omega}$.

We can distinguish two situations depending on the value of~$\omega$.
\begin{itemize}
  \item If~$\omega \geqslant \frac{1}{4}$, then we have~$c_s>0$ for all~$s\in (2,3)$. Moreover, we have that~$c_s > \lambda_s$ for each value of~$s$, which implies from Proposition~\ref{prop-convergence-gamma}
\begin{equation*}
  d_s(\gamma,\gamma_\infty)(t)\leqslant d_s(\gamma_0,\gamma_\infty)e^{-\lambda_st}+L \frac{e^{-\lambda_s t}}{c_s - \lambda_s}.
\end{equation*}
Indeed, in this case, for~$s=3$ and~$\omega=\frac14$, we have~$\lambda_3 - c_3 = 0$, and for~$s \in (2,3)$ the function~$s \mapsto \lambda_s - c_s$ is increasing (it is strictly concave and its derivative at~$s=3$ is~$\frac14(\ln 2-\frac12)>0$). Therefore for~$\omega=\frac14$ we have~$\lambda_s<c_s$ for all~$s\in(2,3)$. Since~$c_s$ is increasing with respect to~$\omega$ (and~$\lambda_s$ does not depend on~$\omega$), this provides the same result for~$\omega\geqslant\frac14$.
\item If~$\omega < \frac{1}{4}$,~$s$ must 
be in~$(2, \bar{s})$ to get~$c_s$ positive. Furthermore the function~$s \mapsto \lambda_s - c_s$ is still increasing on~$(2,3)$, since it is strictly concave and with positive derivative at~$s=3$. Since~$\lambda_3 -c_3 >0$ and~$\lambda_2-c_2<0$, we get that there is a unique value~$s_0$ for which~$\lambda_{s_0} = c_{s_0}$.
Then, Proposition~\ref{prop-convergence-gamma} can be applied. We obtain
\begin{equation*}
  d_{s_0}(\gamma,\gamma_\infty)(t)\leqslant d_{s_0}(\gamma_0,\gamma_\infty)e^{-\lambda_{s_0}t}+Lt e^{-\lambda_{s_0} t}.
\end{equation*}
And for~$s\neq s_0$, we have 
\begin{equation*}
  d_s(\gamma,\gamma_\infty)(t)\leqslant d_s(\gamma_0,\gamma_\infty)e^{-\lambda_st}+L' e^{-\min(\lambda_s,c) t},
\end{equation*}
with~$L' = \frac{L}{|\lambda_s - c|}$.
\end{itemize}

Hence the statement of Theorem~\ref{th-cv-gamma}, together with the precisions given in Remark~\ref{rk-cv-gamma}.
\end{proof}

\section*{Acknowledgements}
The authors are thankful to Florian Patout for the fruitful discussions that helped to achieve this work. A.F. wants to thank the hospitality of the LMA (Université de Poitiers) where most of this research was conducted.

\bibliographystyle{abbrv}
\bibliography{bibli.bib}

\end{document}